\documentclass[a4paper]{amsart}

\usepackage[english]{babel}

\usepackage{amssymb}
\usepackage{amsthm}
\usepackage{mathtools}
\usepackage{dsfont}

\usepackage{lmodern}
\usepackage{microtype}

\usepackage[a4paper, inner=3cm,outer=3cm,top=3.3cm,bottom=3cm]{geometry}

\usepackage{xspace}

\usepackage[autostyle]{csquotes}

\usepackage{tikz}
\usetikzlibrary{matrix,arrows,patterns,decorations.pathreplacing}
\usepackage{pgfplots}

\usepackage[colorlinks, bookmarksdepth=3]{hyperref}

\makeatletter
\@ifpackageloaded{hyperref}{%
\DeclareRobustCommand{\SkipTocEntry}[5]{}}{%
\DeclareRobustCommand{\SkipTocEntry}[4]{}}
\makeatother

\theoremstyle{plain}
\newtheorem{thm}{Theorem}[section]
\newtheorem{prop}[thm]{Proposition}
\newtheorem{cor}[thm]{Corollary}
\newtheorem{lem}[thm]{Lemma}
\newtheorem{conj}[thm]{Conjecture}

\theoremstyle{definition}
\newtheorem{defn}[thm]{Definition}
\newtheorem{ex}[thm]{Example}
\newtheorem{rem}[thm]{Remark}

\newcommand{\cone}{{\mathrm{cone}}}

\newcommand{\lspan}{{\mathrm{span}}}

\newcommand{\Tau}{{\mathcal{T}}}

\DeclareMathOperator{\conv}{conv}

\DeclareMathOperator{\ehr}{ehr}

\newcommand{\N}{{\mathds{N}}}
\newcommand{\Q}{{\mathds{Q}}}
\newcommand{\R}{{\mathds{R}}}
\newcommand{\Z}{{\mathds{Z}}}

\newcommand{\Am}{{\mathcal{A}}}
\newcommand{\Bm}{{\mathcal{B}}}

\newcommand{\Sm}{{\mathcal{S}}}

\newcommand{\eg}{e.\,g.,\xspace}
\newcommand{\floor}[1]{{\rleft\lfloor{#1}\rright\rfloor}}
\newcommand{\lFrac}[1]{{%
\underline{\rleft\{\vphantom{#1}\rright.}%
{#1}%
\underline{\rleft\}\vphantom{#1}\rright.}}}

\newcommand{\ulFrac}[1]{{%
\overline{\underline{\rleft\{\vphantom{#1}\rright.}}%
{#1}%
\overline{\underline{\rleft\}\vphantom{#1}\rright.}}}}
\newcommand{\ie}{i.\,e.,\xspace}
\newcommand{\rleft}{\mathopen{}\mathclose\bgroup\left}
\newcommand{\rright}{\aftergroup\egroup\right}
\newcommand{\vect}[1]{{\boldsymbol{\mathbf{#1}}}}

\newcommand{\set}[1]{\rleft\{ {#1} \rright\}}
\newcommand{\with}{\colon}

\newcommand{\vv}{{\vect{v}}}
\newcommand{\vw}{{\vect{w}}}
\newcommand{\vx}{{\vect{x}}}
\newcommand{\vy}{\vect{y}}
\newcommand{\hst}[2]{h^*_{#1}\rleft(#2\rright)} 

\newcommand{\mm}{\mathfrak{m}}

\newcommand{\kk}{\Bbbk}
\DeclareMathOperator{\reg}{reg}
\DeclareMathOperator{\codim}{codim}
\DeclareMathOperator{\Vol}{Vol}

\pgfplotstableread{
x y
0 4
0.06 4
0.06 3
0.14 3
0.14 2
0.20 2
0.20 3
0.22 3
0.22 2
0.27 2
0.27 5
0.29 5
0.29 4
0.30 4
0.30 6
0.37 6
0.37 5
0.45 5
0.45 4
0.52 4
0.52 3
0.60 3
0.60 5
0.68 5
0.68 4
0.75 4
0.75 3
0.80 3
0.80 5
0.83 5
0.83 4
0.91 4
0.91 3
0.93 3
0.93 6
0.98 6
0.98 5
1.00 5
1.00 4
}\mydata

\title{Ehrhart Theory of Spanning Lattice Polytopes}

\author{Johannes Hofscheier}
\address{Faculty of Mathematics, Otto-von-Guericke Universit\"at Magdeburg,
  Postschlie{\ss}fach 4120, 39106 Magdeburg, Germany.}
\curraddr{}
\email{johannes.hofscheier@ovgu.de}
\thanks{}

\author{Lukas Katth\"an}
\address{Institute of Mathematics, Goethe-University Frankfurt,
  Robert-Mayer-Str. 10, 60325 Frankfurt am Main}
\curraddr{}
\email{katthaen@math.uni-frankfurt.de}
\thanks{}

\author{Benjamin Nill}
\address{Faculty of Mathematics, Otto-von-Guericke Universit\"at Magdeburg,
  Postschlie{\ss}fach 4120, 39106 Magdeburg, Germany.}
\curraddr{}
\email{benjamin.nill@ovgu.de}
\thanks{}


\begin{document}
\selectlanguage{english}

\begin{abstract}
  The key object in the Ehrhart theory of lattice polytopes is the numerator
  polynomial of the rational generating series of the Ehrhart polynomial, called
  $h^*$-polynomial. In this paper we prove a new result on the vanishing of its
  coefficients. As a consequence, we get that $h^*_i =0$ implies $h^*_{i+1}=0$
  if the lattice points of the lattice polytope affinely span the ambient
  lattice. This generalizes a recent result in algebraic geometry due to
  Blekherman, Smith, and Velasco, and implies a polyhedral consequence of the
  Eisenbud--Goto conjecture. We also discuss how this study is motivated by
  unimodality questions and how it relates to decomposition results on lattice
  polytopes of given degree. The proof methods involve a novel combination of
  successive modifications of half-open triangulations and considerations of
  number-theoretic step functions.
\end{abstract}

\maketitle{}

\section{Introduction}
\label{sec:intro}

\subsection{Basics of Ehrhart theory}

The study of Ehrhart polynomials of lattice polytopes is an active area of
research at the intersection of discrete geometry, geometry of numbers,
enumerative combinatorics, and combinatorial commutative algebra. We refer to
\cite{Beck14,Braun2016,Breuer15} for three recent survey articles, as well as to
the book \cite{BR:Computing}. In order to describe our main result, let us
recall the basic notions of Ehrhart theory. We denote by \emph{lattice point}
any element in $\Z^d$. A \emph{lattice polytope} $P \subseteq \R^d$ is the
convex hull of finitely many lattice points, \ie $P = \conv \rleft( \vect{v}_1,
\ldots, \vect{v}_n \rright)$ for $\vect{v}_i \in \Z^d$. To a lattice polytope
$P$, one associates its \emph{Ehrhart function} which counts lattice points in
integral multiples of $P$, \ie $\ehr_P(k) = \rleft| kP \cap \Z^d \rright|$. This
is a polynomial function (see \cite{Ehr62}), called the \emph{Ehrhart
  polynomial} of $P$. Its generating function is known to be a rational function
(see \cite{Sta80})
\begin{align*}
  \sum_{k\ge0} \ehr_P ( k ) t^k = \frac{ h_P^*(t) } {(1-t)^{d+1} }
\end{align*}
where $h_P^*(t) \in \Z_{\ge0}[t]$ is a polynomial of degree $s \in \{0, \ldots,
d\}$, denoted the \emph{$h^*$-polynomial} (or $\delta$-polynomial) of $P$. Its
coefficient vector $(h^*_0, \ldots, h^*_d)$, or $\rleft( h^*_0(P), \ldots,
h^*_d(P) \rright)$ if we want to emphasize that these are the coefficients of
the $h^*$-polynomial of $P$, is the \emph{$h^*$-vector} (or $\delta$-vector) of
$P$. The number $\deg \rleft(P \rright) \coloneqq s$ is called the \emph{degree}
of $P$. For future reference, let us give the basic properties of the
$h^*$-vector of a $d$-dimensional lattice polytope $P$ of degree $s$:

\begin{align}
  h^*_0 &=1 \label{basic1},\\
  h_1^* &= \rleft| P \cap \Z^d \rright| - d - 1 \label{basic2},\\
  h_d^* &= \rleft| P^\circ \cap \Z^d \rright| \label{basic3},\\
  d + 1 - s &= \min \rleft \{ k \in \Z_{>0} \colon (k P)^\circ \cap
              \Z^d \neq \emptyset \rright\} \label{basic4},\\
  \sum_{i=0}^s h^*_i &= \Vol_\Z(P) \label{basic5},
\end{align}

where $\Vol_\Z(P)$ denotes the \emph{normalized volume} of $P$, \ie it equals
$d!$ times the usual Euclidean volume of $P$, and $P^\circ$ denotes the relative
interior of $P$, \ie the topological interior of $P$ in its affine span.

\subsection{Spanning lattice polytopes}

Let us explain what we mean by \enquote{spanning} in the title.

\begin{defn}
  A $d$-dimensional lattice polytope $P \subseteq \R^d$ is called
  \emph{spanning} if any lattice point in $\Z^d$ is an affine integer
  combination of the lattice points in $P$. Equivalently, $P$ is spanning if any
  lattice point in $\Z^{d+1}$ is a linear integer combination of the lattice
  points in $\{ 1 \} \times P$.
\end{defn}

\begin{ex}
  \label{ex:empty}
  Any $1$- and $2$-dimensional lattice polytope is spanning. For any $k \ge 2$,
  the Reeve-simplex is not spanning:
  \begin{align*}
    \conv( \vect{0}, \vect{e}_1, \vect{e}_2, \vect{e}_1 + \vect{e}_2 + k \vect{e}_3 )
  \end{align*}
  where $\vect{e}_1, \vect{e}_2, \vect{e}_3 \in \Z^3$ denotes the standard
  basis. We remark that $h^*_P(t) = 1 + (k-1) t^2$.
\end{ex}

Spanning is a very mild condition for a lattice polytope (for instance, it is
weaker than \enquote{very ample}, cf. \cite{Bru13}). In fact, \emph{any} lattice
polytope is associated to a spanning lattice polytope by a change of the ambient
lattice (replace $\Z^d$ by the lattice affinely spanned by $P \cap
\Z^d$). Especially in toric geometry it is natural to pass from the ambient
lattice to the spanning lattice, \eg for fake weighted projective spaces
(see \cite{Con02,Kas09}) or in the study of $A$-discriminants (see
\cite{Est10,Ito15}).

\smallskip

\subsection{Our main result}

In this work we initiate the study of Ehrhart polynomials of spanning lattice
polytopes.  The main goal of this paper is the following lower bound theorem on
their $h^*$-vectors. It is a direct consequence of a new, general
Ehrhart-theoretic result (Theorem~\ref{thm:nog-vG}), which applies to arbitrary
lattice polytopes, and can be found in Section~\ref{nog-label}.

\begin{thm}
  \label{thm:nog}
  The $h^*$-vector of a spanning polytope $P$ satisfies $h^*_i \ge 1$ for all $i
  = 0, \ldots, \deg(P)$.
\end{thm}
Recall that an $h^*$-vector satisfying the conclusion of Theorem~\ref{thm:nog}
is said to have \emph{no internal zeros} (see, for instance,
\cite{Stanley:Log-concave}). We remark that the $h^{*}$-vector of any lattice
polytope with interior lattice points has no internal zeros by Hibi's lower
bound theorem (see \cite{hibilower}).

\begin{ex}
  \label{ex:nonsp-nog}
  The converse of Theorem~\ref{thm:nog} is not true. In dimension $d \ge 3$,
  there are non-spanning lattice polytopes whose $h^{*}$-vectors have no
  internal zeros. For instance, the lattice polytope $P \coloneqq \conv \rleft(
  \vect{0}, \vect{e}_1, \vect{e}_3, 2 \vect{e}_1 + 4 \vect{e}_2 + \vect{e}_3
  \rright)$ is not spanning and satisfies $h^*_P (t) = 1 + t + 2 t^2$.
\end{ex}

\smallskip

\subsection{Motivation from unimodality questions}

Let us explain why one should view Theorem~\ref{thm:nog} as an example of a
positive result in the quest for unimodality results for $h^*$-vectors of
lattice polytopes. We refer to the survey \cite{Braun2016} for motivation and
background.

We recall that a lattice polytope $P$ is \emph{IDP} (with respect to $\Z^d$) if
for $k \in \Z_{\ge 1}$ any lattice point $\vect{m} \in (k P) \cap \Z^d$ can be
written as $\vect{m} = \vect{m}_1 + \cdots + \vect{m}_k$ for $\vect{m}_1, \dots,
\vect{m}_k \in P \cap \Z^d$. IDP stands for \enquote{integer decomposition
  property}, a condition also referred to as being \emph{integrally-closed}. One
of the main open questions about IDP lattice polytopes (see
\cite{Stanley:Log-concave,OH:Gorenstein,Schepers13}) is whether their
$h^*$-vectors are \emph{unimodal}, \ie their coefficients satisfy $h^*_0 \le
h^*_1 \le \cdots \le h^*_i \ge h^*_{i+1} \ge \cdots \ge h^*_s$ for some $i \in
\{0, \ldots, s\}$. Theorem~\ref{thm:nog} is a modest analogue of this
conjecture. Clearly, IDP implies spanning, and unimodality implies no internal
zeros.

\begin{ex}
  \label{gabriele}
  The $5$-dimensional lattice simplex with vertices
  \begin{align*}
    \vect{0}, \vect{e}_1, \ldots, \vect{e}_4, \vect{v} \coloneqq
    5\rleft( \vect{e}_1 + \ldots + \vect{e}_4 \rright) + 8 \vect{e}_5
  \end{align*}
  is spanning with $h^*$-vector $(1,1,2,1,2,1)$, \ie not unimodal. We have used
  the software \texttt{polymake} (see \cite{polymake}) to compute the
  $h^*$-vector and the lattice points contained in the simplex which are exactly
  the vertices and the point $\vect{w} \coloneqq 2 \rleft( \vect{e}_1 + \ldots +
  \vect{e}_4 \rright) + 3 \vect{e}_5$. As $\vect{e}_5 = 2 \vect{v} - 5
  \vect{w}$, it follows that the simplex is spanning.
\end{ex}

From the viewpoint of commutative algebra, it was already evident that IDP
implies no internal zeros. Theorem~\ref{thm:nog} provides a new combinatorial
proof of this fact.  Indeed, the Ehrhart ring associated to an IDP polytope $P$
(cf.~\cite[Section~4]{BG}) is standard graded and Cohen--Macaulay, so its
quotient modulo a linear system of parameters yields a standard graded Artinian
algebra whose Hilbert series equals $h^*_P(t)$, which clearly has no internal
zeros.  Let us remark that for spanning lattice polytopes it is unclear whether
such an algebraic proof exists, the difficulty being that the Ehrhart ring of
non-IDP lattice polytopes is not standard graded.

Another conjecture of interest is Oda's question whether every smooth lattice
polytope is IDP \cite{Gub12}. Here, a lattice polytope is \emph{smooth} if the
primitive edge directions at each vertex form a lattice basis. As smooth
polytopes are spanning, Theorem~\ref{thm:nog} shows that the condition of having
no internal zeros cannot be used to distinguish between smoothness and IDP.

\smallskip

The methods of the proof of Theorem~\ref{thm:nog} combine modifications of
half-open triangulations and considerations of number-theoretical step
functions. We hope that these methods will also be fruitful to prove stronger
inequalities on the coefficients of $h^*$-polynomials. Let us remark that
Schepers and van Langenhoven (see \cite{Schepers13}) suggested that a successive
change of lattice triangulations should be essential in achieving new
unimodality results in Ehrhart theory. In this sense, our results and methods
could be seen as a first implementation of their proposed approach.

\subsection{Organization of the paper}

In Section 2 we explain how Theorem~\ref{thm:nog} implies a consequence of the
Eisenbud--Goto conjecture from commutative algebra in this polyhedral setting and
give some combinatorial consequences.  Theorem~\ref{thm:nog} can be seen as a
generalization of a recent result on the vanishing of the second coefficient of
the $h^*$-polynomial (see \cite{Blek16}). This observation and applications to
decomposition results of lattice polytopes of given degree are discussed in
Section 3. In Section 4 we recall the language of half-open decompositions and
describe how Theorem~\ref{thm:nog} follows from Theorem~\ref{thm:nog-vG}, a
general result in Ehrhart theory. Section~\ref{sec:bigproof} contains the proof
of Theorem~\ref{thm:nog-vG}.

\section*{Acknowledgments}
The authors would like to thank Christian Haase for several inspiring
discussions and Gabriele Balletti for his supply of computational data, \eg
Example~\ref{gabriele}, as well as many fruitful discussions. The second author
thanks KTH Stockholm and Stockholm University for their hospitality. The third
author is an affiliated researcher with Stockholm University and partially
supported by the Vetenskapsr{\aa}det grant~NT:2014-3991.

\section{Application 1: Polyhedral Eisenbud--Goto}
\label{sec:appl2}

One of the original motivations of the present work is a connection with the
famous Eisenbud--Goto conjecture from commutative algebra, which we explain in
this section.  For the algebraic concepts used in this chapter, we refer the
reader to the monographs by Eisenbud \cite{eisenbud} or Brodmann and Sharp
\cite{BS}.  Let us recall the statement of the conjecture:
\begin{conj}[Eisenbud--Goto conjecture \cite{EG}]\label{conj:EG}
  Let $\kk$ be a field and let $S = \kk[X_1,\dotsc, X_{n}]$ be a polynomial ring
  with the standard grading, and let $I \subseteq S$ be a homogeneous prime
  ideal. Then
  \begin{equation}\label{eq:EG}
    \reg(S/I) \leq \deg(S/I) - \codim(S/I)
  \end{equation}
\end{conj}
Here, $\reg(S/I)$ denotes the \emph{(Castelnuovo--Mumford) regularity} of $S/I$,
which is defined as
\begin{align*}
  \reg(S/I) \coloneqq \sup\set{i + j \with i,j \in
  \N_0, H_\mm^i(S/I)_j \neq 0} \text{,}
\end{align*}
where $\mm = (X_1, \dotsc, X_n)$ is the maximal homogeneous ideal of $S$ and
$H_\mm^i(S/I)_j$ denotes the $j$-th homogeneous component of the $i$-th local
cohomology module of $S/I$ with support in $\mm$.  Further, the \emph{degree} of
$S/I$, denoted by $\deg(S/I)$, can be defined as $(\dim (S/I)-1)!$ times the
leading coefficient of the Hilbert polynomial of $S/I$.  Moreover, $\codim(S/I)
= \dim_\kk (S/I)_1 - \dim(S/I)$ is the codimension of $S/I$ (inside its linear
hull).

Very recently, McCullough and Peeva \cite{MP:Counterexample} found
counterexamples to this conjecture.  However, we are going to show that a
certain consequence of it is nevertheless true.  Let $P \subseteq \R^d$ be a
$d$-dimensional lattice polytope, and let $\kk$ be an algebraically closed field
of characteristic $0$.  We denote by $\kk[P]$ the {\em toric ring} generated by
the lattice points in $P$, \ie the subalgebra of $\kk[Y_0, \dotsc, Y_{d}]$
generated by the monomials
\begin{align*}
  Y_0 \cdot \prod_i Y_{i}^{v_i} \qquad\text{with } \vv =
  \rleft( v_1, \dotsc, v_d \rright) \in P \cap \Z^d\text{.}
\end{align*}

The algebraic invariants on the right-hand side of \eqref{eq:EG} have a
combinatorial interpretation for $S/I = \kk[P]$:
\begin{align*}
  \deg(\kk[P]) &= \Vol_{\Gamma_{P}}(P)  \\
  \codim(\kk[P]) &= |P \cap \Z^{d}| - (d+1) \text{.}
\end{align*}
Here, $\Vol_{\Gamma_{P}}$ is the volume form normalized with respect to the
affine lattice generated by the lattice points in $P$. In particular, if $P$ is
spanning, then this simply equals $\Vol_\Z(P)$.

The regularity of $\kk[P]$ does not have a direct combinatorial interpretation.
However, if $P$ is spanning and $\vv = \rleft( v_{0}, \ldots, v_{d} \rright)$ is
an interior lattice point of the cone $C$ over $P$, then $H_\mm^{d+1}( \kk[P]
)_{-v_0} \neq 0$, cf. \cite[Theorem~5.6]{ss} or
\cite[Proposition~4.2]{depthholes}. Thus, if we let $r \in \Z_{>0}$ be the
minimal value for the first coordinate of an interior lattice point in $C$, \ie
the minimal number such that the multiple $rP$ of $P$ has an interior lattice
point, then it holds that $\reg(\kk[P]) \geq d+1 - r$.

In conclusion, the following proposition is a consequence of
Conjecture~\ref{conj:EG}:
\begin{prop}
  \label{prop:pEG}
  Let $P \subseteq \R^{d}$ be a ${d}$-dimensional spanning lattice
  polytope. Then the following holds:
  \begin{equation}\label{eq:vol}
    |P \cap \Z^{d}| \leq \Vol_\Z(P) +
    \min\set{k \in \Z_{> 0} \with (kP)^\circ \cap
      \Z^{d} \neq \emptyset}
  \end{equation}
  or equivalently,
  \begin{equation}\label{eq:vol2}
    h^*_1 + \deg(P) \le \Vol_\Z(P)\text{.}
  \end{equation}
\end{prop}
As the Eisenbud--Goto conjecture has been disproven in \cite{MP:Counterexample}, we show that this
inequality is also a consequence of our main result Theorem~\ref{thm:nog}. Let
us remark that \eqref{eq:vol2} is sharp for every value of $\deg(P)$, as can be
seen by considering the lattice simplices $\conv( \vect{e}_{1}, \ldots,
\vect{e}_{d}, -\vect{e}_{1} - \ldots - \vect{e}_{d})$.

\begin{proof}
  Equations \eqref{eq:vol} and \eqref{eq:vol2} are equivalent by the properties
  \eqref{basic2} and \eqref{basic4} of $h^*$-vectors. By properties
  \eqref{basic1} and \eqref{basic5}, we can reformulate \eqref{eq:vol2} as
  \begin{align*}
    \deg(P) \leq 1 + \sum_{i=2}^{\deg(P)} h^*_i(P) \text{.}
  \end{align*}
  This equation holds as $h^*_i(P) \ge 1$ for $2 \leq i \leq \deg(P)$ by
  Theorem~\ref{thm:nog}.
\end{proof}

\begin{ex}
  In dimension $5$ there exists a non-spanning lattice simplex with binomial
  $h^*$-polynomial $1 + t^3$ (see, for instance, \cite[end of Section~2]{Hibi11}
  or \cite[paragraph below Lemma~1.3]{HT:LowerBounds}). Hence, the left side in
  Equation~\eqref{eq:vol} equals $6$, while the right side equals $2+3=5$. This
  shows that the spanning assumption cannot be dropped in
  Proposition~\ref{prop:pEG}.
\end{ex}

Proposition~\ref{prop:pEG} has an immediate combinatorial consequence. Let us
recall that two polytopes in $\R^d$ are \emph{affinely equivalent} if they are
mapped onto each other by an affine-linear automorphism of $\R^d$. Moreover, we
say that two affinely equivalent lattice polytopes in $\R^d$ are
\emph{unimodularly equivalent} if such an affine-linear automorphism maps $\Z^d$
to $\Z^d$. In fixed dimension there are only finitely many lattice polytopes of
bounded volume up to unimodular equivalence (see \cite{LZ91}). Batyrev showed
more generally that there are only finitely many lattice polytopes (of arbitrary
dimension) of given degree and of bounded volume up to unimodular equivalence
and lattice pyramid constructions (see \cite{Bat06}). Here, $P \subseteq \R^d$
is a {\em lattice pyramid} if $P$ is unimodularly equivalent to $\conv \rleft(
\{ \vect{0} \}, \{ 1 \} \times P' \rright)$ for some lattice polytope $P'
\subseteq \R^{d-1}$. We recall that $h^*$-vectors of lattice polytopes are
invariant under lattice pyramid constructions (see, for instance,
\cite[Theorem~2.4]{BR:Computing}).

There exist (non-spanning) lattice polytopes of normalized volume $2$ for each
degree, none of them being a lattice pyramid of the other (see
\cite{Hibi11,HT:LowerBounds}). Such a situation cannot happen for spanning
lattice polytopes, since by Equation~\eqref{eq:vol2} a bound on the normalized
volume also implies a bound on the degree.

\begin{cor} 
  \label{peg-cons}
  There are only finitely many spanning lattice polytopes of given normalized
  volume (and arbitrary dimension) up to unimodular equivalence and lattice
  pyramid constructions.
\end{cor}

\begin{rem}
  While the generalization in \cite{Nil08} of Batyrev's result might suggest
  this, we remark that it is not enough to fix $h^*_1$ and the degree of a
  spanning lattice polytope in order to bound its volume. To see this, we
  consider the three-dimensional lattice polytope $P$ with vertices
  \begin{align*}
    \vect{0}, \vect{e}_1, \vect{e}_2, -\vect{e}_3,
    \vect{e}_1 + \vect{e}_2 + a \vect{e}_3
  \end{align*}
  with $a \in \Z_{\ge 2}$. Then $P$ is spanning of (normalized) volume $a+1$
  where the only lattice points in $P$ are its vertices, so, $h^*_1 = 1$ and $s
  = 2$.
\end{rem}

\section{Application 2: On the vanishing of \texorpdfstring{$h^*$}{h*}-coefficients}
\label{sec:appl1}

\subsection{Passing to spanning lattice polytopes}

Let $P \subseteq \R^d$ be a $d$-dimensional lattice polytope (with respect to
$\Z^d$). Let us denote by $\Gamma_P$ the affine sublattice in $\Z^d$ generated
by $P \cap \Z^d$, \ie the set of all integral affine combinations of $P \cap
\Z^{d}$. We define the {\em spanning polytope $\tilde{P}$ associated to} $P$ as
the lattice polytope given by the vertices of $P$ with respect to the lattice
$\Gamma_P$.

Let us say that two lattice polytopes $P$, $P'$ are \emph{lattice-point
  equivalent} if there is an affine-linear automorphism of $\R^d$ mapping $P$ to
$P'$ such that the lattice points in $P$ map bijectively to the lattice points
in $P'$. In particular $P$ and $\tilde{P}$ are lattice-point
equivalent. Clearly, unimodularly equivalent implies lattice-point equivalent
implies affinely equivalent, however, none of the converses is generally
true. As $\Vol_\Z ( \tilde{P} ) \le \Vol_\Z(P)$, Corollary~\ref{peg-cons} has the following
Corollary~\ref{cor:vol-finite} as an immediate consequence for lattice polytopes
that are not necessarily spanning. For this, we call $P$ a \emph{lattice-point
  pyramid} if there is a facet of $P$ that contains all lattice points of $P$
except for one. Note that lattice pyramids are lattice-point pyramids, but not
vice versa.

\begin{cor}
  \label{cor:vol-finite}
  There are only finitely many lattice polytopes of given normalized volume (and
  arbitrary dimension) up to lattice-point equivalence and lattice-point pyramid
  constructions.
\end{cor}

We remark that this corollary can be also obtained from
\cite[Corollary~3.9]{Arnau}.

\subsection{Bounding the degree of the spanning lattice polytope}

As $h_1^*$ equals the number of lattice points minus dimension minus one, we get
$h^*_1(\tilde{P}) = h^*_1(P)$. For $i \ge 2$, it holds $h^*_i(\tilde{P}) \le
h^*_i(P)$. This follows from the description of $h^*_i$ as the number of lattice
points in half-open parallelepipeds, see Equation~\eqref{h-stern-equ} in
Section~\ref{subsec:hot}. In particular, $\deg(\tilde{P}) \le \deg(P)$.

The previous considerations show that Theorem~\ref{thm:nog} has the following
corollary.

\begin{cor}
  \label{main-cor}
  If $P$ is a lattice polytope with $h^*_i ( P ) = 0$, then $\deg(\tilde{P}) \le
  i-1$.
\end{cor}

In other words, the first zero in the $h^*$-vector of $P$ bounds the degree of
its spanning polytope.

\begin{rem}
  \label{heins}
  For $i = 1$, Corollary~\ref{main-cor} is even an equivalence. We give an
  elementary proof. Recall that a lattice polytope is an
  \emph{empty} lattice simplex if $|P \cap \Z^d| = d+1$, equivalently, $h^*_1 (
  P ) = 0$. Moreover, a lattice polytope $P$ is a \emph{unimodular simplex} if
  its vertices form an affine lattice basis. Equivalently, $\Vol_\Z( P ) = 1$,
  respectively, $\deg( P ) = 0$. We observe that a spanning lattice polytope is
  an empty simplex if and only if it is a unimodular simplex. In particular,
  $h^*_1(P)=0$ is equivalent to $\deg( \tilde{P} ) = 0$.
\end{rem}

For each $i \ge 2$, there exist empty lattice simplices $P$ with $h^*_i = 1$
(see \cite{Hibi11,HT:LowerBounds}). Hence, the converse of
Corollary~\ref{main-cor} fails for $i \ge 2$.

\subsection{The vanishing criterion by Blekherman, Smith, and Velasco}

While Corollary~\ref{main-cor} describes a necessary condition on the vanishing
of $h^*_i$, it is a natural question how to strengthen it to get an equivalence
also for $i \ge 2$. Recently such a criterion was proven for $i=2$ (see
\cite{Blek16}). In order to describe this result, let us denote a lattice
polytope $P \subseteq \R^d$ as \emph{$i$-IDP} if any lattice point $\vect{m} \in
(i P) \cap \Z^d$ can be written as $\vect{m} = \vect{m}_1 + \ldots + \vect{m}_i$
for $\vect{m}_1, \ldots, \vect{m}_i \in P \cap \Z^d$.

\begin{prop}[{\cite[Proposition~6.6]{Blek16}}]
  \label{blekherman}
  A lattice polytope $P$ satisfies $h^*_2(P)=0$ if and only if $\deg(\tilde{P})
  \le 1$ and $P$ is $2$-IDP.
\end{prop}

This is a reformulation of \cite[Proposition~6.6]{Blek16} in our notation. The
hard non-combinatorial part of their proof that relies on results from real and
complex algebraic geometry is the statement $h^*_2(P)=0$ implies $\deg(\tilde P)
\le 1$. This follows now from Corollary~\ref{main-cor} for $i=2$. The authors of
\cite{Blek16} communicated to us another purely combinatorial proof that relies
on the classification of lattice polytopes of degree one (see \cite{BN07}). We
remark that such a classification is not known for lattice polytopes of higher
degree.

\smallskip

The sufficient condition on the vanishing of $h^*_2(P)$ in
Proposition~\ref{blekherman} easily generalizes.

\begin{prop}
  \label{prop:hi-vanish}
  If $\deg(\tilde{P}) \le i-1$ and $P$ is $i$-IDP, then $h^*_i(P)=0$.
\end{prop}
\begin{proof}
  We show the contraposition, so assume $h^*_i(P) > 0$. Then there exists a
  lattice point of height $i$ in some half-open parallelepiped of a given
  half-open triangulation of $P$, we refer to Section~\ref{subsec:hot} for more
  details. As $P$ is $i$-IDP, the lattice point is also contained in the
  sublattice $\Gamma_P$, hence, $h^*_i(\tilde{P}) > 0$, and thus $\deg(\tilde P)
  > i-1$.
\end{proof}

\begin{rem}
  For $i \ge 3$ it is not true that $h^*_i(P)=0$ implies that $P$ is
  $i$-IDP. There exists a spanning (even very ample) lattice polytope $P'
  \subseteq \R^3$ with $h^*$-vector $h^*(P')=(1,4,5,0)$ such that a lattice
  point in $2P'$ is not a sum of two lattice points in $P'$ (see
  \cite{Bru13,Poly16}). This lattice polytope can be constructed as the
  Minkowski sum of the Reeve-simplex $R_4 \coloneqq \conv( \vect{0}, \vect{e}_1,
  \vect{e}_2, \vect{e}_1 + \vect{e}_2 + 4 \vect{e}_3 ) \subseteq \R^3$ and the
  edge $\conv( \vect{0}, \vect{e}_3 ) \subseteq \R^3$ (see
  \cite{Ogata}). Therefore, the lattice pyramid $P \subseteq \R^4$ over $P'$ is
  a $4$-dimensional spanning lattice polytope of degree $2$ that is not $3$-IDP,
  as the lattice point $2 \vect{e}_0 + \vect{e}_1 + \vect{e}_2 + 3 \vect{e}_3
  \in 3 P \cap \Z^4$ cannot be written as the sum of three lattice points in
  $P$.
\end{rem}

\subsection{Generalizing results on the degree of lattice polytopes}

In \cite{BN07,Nil08,HNP09} it was investigated how lattice polytopes of small
degree can be decomposed into lower-dimensional lattice polytopes. This question
is partly motivated by applications in algebraic geometry
\cite{Dickenstein-Sandra, Dickenstein-Nill, Adjunction-Paper,
  Ito15}. Corollary~\ref{main-cor} allows to generalize these results up to a
change of lattice.  For this, let us recall that $P \subseteq \R^d$ is called a
\emph{Cayley polytope} of lattice polytopes $P_1, \ldots, P_k \subseteq \R^m$ if
$k \ge 2$ and $P$ is unimodularly equivalent to $\conv(P_1 \times \vect{e}_1,
\ldots, P_k \times \vect{e}_k) \subseteq \R^m \times \R^k$ where $\vect{e}_1,
\ldots, \vect{e}_k$ denotes the standard basis of $\Z^k$. In particular, note
that the lattice points of a Cayley polytope lie on two parallel affine
hyperplanes of lattice distance one.

\begin{cor}
  \label{cor-cayley}
  Let $P$ be a $d$-dimensional lattice polytope with $h^*_{i+1} = 0$. If $d >
  \frac{i^2+19i-4}{2} \eqqcolon d'$, then $\tilde{P}$ is a Cayley polytope of
  lattice polytopes in dimension at most $d'$. In this case, every lattice point
  in $P$ lies on one of two parallel hyperplanes.
\end{cor}
\begin{proof}
  By, Corollary~\ref{main-cor}, $\deg(\tilde{P}) \le i$. Now, we apply
  \cite[Theorem~1.2]{HNP09} to $\tilde{P}$.
\end{proof}

\begin{rem}
  Let us briefly discuss the relation of the results of this section to the
  study of point configurations of small \emph{combinatorial degree}, \ie the
  maximal degree of the $h$-vector of lattice triangulations of $P$. We refer to
  \cite{Arnau} for terminology and background. Let us observe that the
  $h$-vector of a lattice triangulation $\Tau$ of $P$ has no internal
  zeros. This can be deduced from the fact that the $h$-vector is an
  $M$-sequence (see \cite{Bruns}); an alternative, direct proof can also be
  given using Lemma~\ref{lem:simpl-compl}. Now, it follows from the
  Betke--McMullen formula (see \cite{Betke}) that $h^*_{i+1} (P) = 0$ implies
  $h_{i+1}( \Tau ) = 0$. Hence, the combinatorial degree of $P$ is bounded by
  $i$ if $h^*_i = 0$. This shows that Corollary~\ref{cor-cayley} sharpens in
  this case the conclusion in \cite[Theorem~A]{Arnau} which only guaranteed a
  so-called \enquote{weak Cayley} condition.
\end{rem}

\section{Ehrhart Theory and Half-open Triangulations}
\label{sec:ehrhart}

\subsection{Half-open triangulations}
\label{subsec:hot}
In this subsection let $P \subseteq \R^d$ be a $d$-dimensional lattice
polytope. The polynomial $h^*_P$ can be computed by means of the \emph{cone $C$
  over $P$}, \ie $C = \cone \rleft( \{ 1 \} \times P \rright) \subseteq
\R^{d+1}$, equipped with a triangulation which we now outline. For details and
references on Ehrhart Theory, we refer to \cite{BR:Computing}. Our approach is
in the spirit of \cite{KV:Computing} (see also \cite{HNP:LattPoly}).

In this paper, by a \emph{triangulation} $\Tau$ of $C$, we mean a regular
triangulation of $C$ such that the primitive ray generators of every face of the
triangulation are contained in the affine hyperplane $\{1\} \times \R^d$. The
set of faces of dimension $k$ we denote by $\Tau^{(k)}$.

A point $\xi \in \R^{d+1}$ is called \emph{generic} with respect to a
triangulation $\Tau$ of $C$, if it is not contained in any of the linear
subspaces generated by the faces in $\Tau^{(d)}$.

We define
\begin{align*}
  \Upsilon_C & \coloneqq \{ \Tau \; \text{triangulation of} \; C \}
               \text{,} \\
  \Xi_C & \coloneqq \{ \xi \in C \; \text{generic with respect to
          any} \; \Tau \in \Upsilon_C \} \text{.}
\end{align*}

The set of primitive generators in $\Z^{d+1}$ of the extremal rays of a
polyhedral cone $\sigma \subseteq \R^{d+1}$, we denote by $\sigma^{(1)}$.

\begin{defn}
  \label{def:hot}
  A \emph{half-open triangulation} of $C$ consists of a choice $\rleft( \Tau,
  \xi \rright) \in \Upsilon_C \times \Xi_C$. For every maximal cell $\sigma \in
  \Tau^{(d+1)}$ the corresponding \emph{half-open cell} $\sigma[ \xi )$ is given
  as follows: Write $\xi = \sum_{\vect{v} \in \sigma^{(1)}} \lambda_{\vect{v}}
  \vect{v}$ for $\lambda_{\vect{v}} \in \R \setminus \{ 0 \}$ and set $I_\xi
  \rleft( \sigma \rright) \coloneqq \set{ \vect{v} \in \sigma^{(1)} \with
    \lambda_{\vect{v}} < 0 }$. Then
  \begin{align*}
    \sigma \rleft[ \xi \rright) = \set{ \sum_{\vect{v} \in
    \sigma^{(1)}} \mu_{\vect{v}} \vect{v} \with \mu_{\vect{v}} \in
    \R_{\ge0}, \; \mu_{\vect{v}} > 0 \; \text{for all} \; \vect{v} \in
    I_\xi \rleft( \sigma \rright) } \text{.}
  \end{align*}
\end{defn}

The proofs of the following results in Section~\ref{subsec:hot} are standard and
can be done as in \cite{HNP:LattPoly}.

\begin{prop}
  \label{prop:hoc-tile-c}
  Let $\rleft( \Tau, \xi \rright) \in \Upsilon_C \times \Xi_C$ be a half-open
  triangulation of $C$. The half-open cells $\sigma \rleft[ \xi \rright)$ for
  $\sigma \in \Tau^{(d+1)}$ yield a partition of $C$, \ie we have a disjoint
  union
  \begin{align*}
    C = \bigcup_{\sigma \in \Tau^{(d+1)}} \sigma \rleft[ \xi \rright) \text{.}
  \end{align*}
\end{prop}

\begin{defn}
  \label{def:hop}
  Let $\rleft( \Tau, \xi \rright) \in \Upsilon_C \times \Xi_C$ be a half-open
  triangulation of $C$. The \emph{half-open fundamental parallelepiped}
  $\Pi_\sigma \rleft[ \xi \rright)$ of a half-open cell $\sigma \rleft[ \xi
  \rright)$ for $\sigma \in \Tau^{(d+1)}$ is given by
  \begin{align*}
    \Pi_\sigma \rleft[ \xi \rright) = \set{ \sum_{\vect{v} \in
    \sigma^{(1)}} \lambda_{\vect{v}} \vect{v} \with \lambda_{\vect{v}}
    \in [0,1[ \; \text{for} \; \vect{v} \not \in I_\xi ( \sigma ),
    \lambda_{\vect{v}} \in ]0,1] \; \text{for} \; \vect{v} \in I_\xi (
    \sigma ) } \text{.}
  \end{align*}
\end{defn}

\begin{rem}
  Let $\rleft( \Tau, \xi \rright) \in \Upsilon_C \times \Xi_C$ be a half-open
  triangulation of $C$ and take $\sigma \in \Tau^{(d+1)}$. If $\xi \in
  \sigma^\circ$, then note that $\Pi_\sigma \rleft[ \xi \rright)$ is the usual
  half-open parallelepiped, \ie
  \begin{align*}
    \Pi \rleft[ \xi \rright) = \set{ \sum_{\vect{v} \in \sigma^{(1)}} \lambda_{\vect{v}}
    \vect{v} \with \lambda_{\vect{v}} \in [ 0, 1 [ } \text{.}
  \end{align*}
\end{rem}

\begin{prop}
  \label{prop:hop-tile-hoc}
  Let $\rleft( \Tau, \xi \rright) \in \Upsilon_C \times \Xi_C$ be a half-open
  triangulation and fix a half-open cell $\sigma \rleft[ \xi \rright)$ for
  $\sigma \in \Tau^{(d+1)}$. The translates of $\Pi \rleft[ \xi \rright)$ by
  vectors in $\mathcal{M} \coloneqq \sum_{\vect{w} \in \sigma^{(1)}} \Z_{\ge0}
  \vect{w}$ yield a partition of the half-open cell $\sigma \rleft[ \xi
  \rright)$, \ie we have a disjoint union
  \begin{align*}
    \sigma \rleft[ \xi \rright) = \bigcup_{\vect{v} \in \mathcal{M}} \vect{v} +
    \Pi \rleft[ \xi \rright) \text{.}
  \end{align*}
\end{prop}

\begin{defn}
  \label{def:hstar-fct}
  We define a map
  \begin{align*}
    h^* \colon \rleft( C \cap \Z^{d+1} \rright) \times \Upsilon_C
    \times \Xi_C \to \Z_{\ge0} \text{,}
  \end{align*}
  as follows: For given $\vect{v} \in \rleft( C \cap \Z^{d+1} \rright)$, $\Tau
  \in \Upsilon_C$ and $\xi \in \Xi_C$, there is exactly one $\sigma \in
  \Tau^{(d+1)}$ such that $\vect{v}$ is contained in the half-open cell $\sigma[
  \xi )$.  There is a unique $\ulFrac{\vect{v}}_{\Tau,\xi} \in \Pi_\sigma
  \rleft[ \xi \rright) \cap \Z^{d+1}$ such that $\vect{v} -
  \ulFrac{\vect{v}}_{\Tau,\xi} \in \sum_{\vect{w} \in \sigma^{(1)}} \Z_{\ge0}
  \vect{w}$. Then $h^*_{\Tau, \xi} ( \vect{v} )$ is, by definition, equal to the
  first coordinate of $\ulFrac{\vect{v}}_{\Tau,\xi}$.
\end{defn}

\begin{figure}[!ht]
  \centering
  \begin{tikzpicture}
    \draw (.1,0) -- (2.6,2.5);
    \draw (-.1,0) -- (-2.6,2.5);
    \draw (-.1,0) -- (-.1,2.5) node[left] {$\sigma_{1} \rleft[ \xi \rright)$};
    \node[above] at (0,2.5) {$C$};
    \draw[dashed] (.1,0) -- (.1,2.5) node[right] {$\sigma_{2}\rleft[ \xi \rright)$};
    \draw[thick] (-.1,0) -- (-1.1,1);
    \draw[thick] (-.1,0) -- (-.1,1);
    \draw[dashed] (-.1,1) -- ( -1.1,2);
    \draw[dashed] (-1.1,1) -- (-1.1,2);
    
    \draw[dashed] (.1,1) -- (1.1,2);
    \draw[thick] (.1,0) -- (1.1,1);
    \draw[thick] (1.1,1) -- (1.1,2);
    
    \fill (-1.7,2) circle (1pt) node[above] {$\xi$};
    
    \fill (1.6,3) circle (2pt) node[right] {$\vv$};
    
    \draw[-latex,very thick] (1.6,3) -- (1.6,2);
    \draw[-latex,very thick] (1.6,2) -- (.6,1) node[below right] {$\ulFrac{\vv}_{\Tau,\xi}$};
    
    \fill (-.1,0) circle (2pt);
    \fill[fill=white,draw=black] (.1,0) circle (2pt);
    
    \fill[fill=white,draw=black] (-1.1,1) circle (2pt);
    \fill[fill=white,draw=black] (-1.1,2) circle (2pt);
    \fill[fill=white,draw=black] (-.1,1) circle (2pt);
    \fill[fill=white,draw=black] (.1,1) circle (2pt);
    \fill (1.1,1) circle (2pt);
    \fill[fill=white,draw=black] (1.1,2) circle (2pt);
    
    \fill (-.6,1) circle (2pt);
    \fill (.6,1) circle (2pt);
  \end{tikzpicture}
  \caption{Illustration of Definition~\ref{def:hstar-fct} for $P = [-2,2]$
    (here, $h^{*}_{\Tau,\xi}(\vv) = 1$).}
  \label{fig:tulip}
\end{figure}

For fixed $\Tau \in \Upsilon_C$ and $\xi \in \Xi_C$, the \emph{$h^*$-polynomial}
of $P$ is given by
\begin{equation}
  \label{h-stern-equ}
  h^*_P ( t ) = \sum_{k=0}^s h_k^* t^k = \sum_{\sigma \in \Tau^{(d+1)}}
  \sum_{\vect{v} \in \Pi_\sigma[\xi) \cap \Z^{d+1}}
  t^{h^*_{\Tau, \xi}\rleft(\vect{v}\rright)} \text{.}
\end{equation}
From this equality it is evident that the coefficients $h^*_k$ are non-negative
integers. In particular, we observe that
\begin{equation}
  \label{eq:h-star-non-vanish}
  \set{ k = 0, \ldots, s \with h_k^* \neq 0 } = \set{h^*_{\Tau,\xi} ( \vect{v} )
    \with \vect{v} \in C \cap \Z^{d+1} } \text{.}
\end{equation}

\subsection{A general Ehrhart-theoretic result and the proof of Theorem
  \ref{thm:nog}}
\label{nog-label}
Theorem~\ref{thm:nog} is an immediate consequence of the following main result
of this paper.

\begin{thm}
  \label{thm:nog-vG}
  Let $P \subseteq \R^d$ be a $d$-dimensional lattice polytope and let $C
  \subseteq \R^{d+1}$ be the cone over it. Let $\Gamma_P$ be the sublattice of
  $\Z^{d+1}$ spanned by the lattice points in $\{1 \} \times P$. Then for every
  $\vect{v} \in \Z^{d+1}$ and all tuples $\rleft( \Tau_0, \xi_0 \rright) \in
  \Upsilon_C \times \Xi_C$, there exist nonnegative integers $a_{\vect{v}} \le
  b_{\vect{v}}$ (independent of the choice $\rleft( \Tau_0, \xi_0 \rright)$)
  such that
  \begin{align*}
    \rleft[ a_{\vect{v}}, b_{\vect{v}} \rright] \cap \Z  =
    h^*_{\Tau_0,\xi_0} \rleft( C \cap \rleft( \vect{v} +
    \Gamma_P \rright) \rright)\text{.}
  \end{align*}
\end{thm}

The proof of Theorem~\ref{thm:nog-vG} will be developed in
Section~\ref{sec:bigproof}. Let us show here how to use Theorem~\ref{thm:nog-vG}
to prove Theorem~\ref{thm:nog}.

\begin{proof}[Proof of Theorem~\ref{thm:nog}]
  As in the statement of Theorem~\ref{thm:nog-vG}, let $\Gamma_P$ be the
  sublattice spanned by the lattice points in $\{1\} \times P$. Since $P$ is
  spanning, we obtain $\Gamma_P = \Z^{d+1}$. The statement follows from Equation
  \eqref{eq:h-star-non-vanish} by applying Theorem~\ref{thm:nog-vG} with
  $\vect{v} \coloneqq \vect{0}$.
\end{proof}

\begin{rem}
  \label{rem:hstar-interval}
  One can also interpret Theorem~\ref{thm:nog-vG} as follows. We use the
  notation from that theorem. Assume that $h^*_b = h^*_B = 0$ for two integers
  $b < B$ such that $h^*_k \neq 0$ for all $k = b+1, \ldots, B-1$. Fix $\rleft(
  \Tau, \xi \rright) \in \Upsilon_C \times \Xi_C$ and take a vector $\vect{v}
  \in C \cap \Z^{d+1}$ with $b < h^*_{\Tau, \xi} \rleft( \vect{v} \rright) <
  B$. Let $\vect{v}_1, \ldots, \vect{v}_a, \vect{v}'_1, \ldots, \vect{v}'_A \in
  \rleft( \{1\} \times P \rright) \cap \Z^{d+1}$ such that $\vect{v}' \coloneqq
  \vect{v} + \sum_{i=1}^a \vect{v}_i - \sum_{j=1}^A \vect{v}'_j \in C$. Then $b
  < h^*_{\Tau,\xi} \rleft( \vect{v}' \rright) < B$, \ie the lattice points in
  $C$ that can be reached from $\vect{v}$ by adding or subtracting lattice
  points from $\{1\} \times P$ contribute only to the $h^*$-coefficients with
  index in the interval $]b,B[$.
\end{rem}

\section{Proof of Theorem \ref{thm:nog-vG}}
\label{sec:bigproof}
\subsection{Overview}
\label{sec:overview-proof}

We give an overview of the proof of Theorem~\ref{thm:nog-vG}. We use the
notation from that theorem with $\Gamma \coloneqq \Gamma_P$. We start with the
following observation.
\begin{lem}
  For all pairs $\rleft( \Tau_0, \xi_0 \rright), \rleft( \Tau, \xi \rright) \in
  \Upsilon_C \times \Xi_C$, it holds that
  \begin{align*}
    h^*_{\Tau_0,\xi_0} \rleft( C \cap
    \rleft( \vect{v} + \Gamma \rright) \rright) = h^*_{\Tau,\xi} \rleft( C \cap
    \rleft( \vect{v} + \Gamma \rright) \rright) \text{.}
  \end{align*}
\end{lem}

\begin{proof}
  We note that
  \begin{align*}
    \sum_{k = 0}^\infty \rleft| \rleft( \{ k \} \times \rleft( k P \rright) \rright) \cap 
    \rleft( \vect{v} + \Gamma \rright) \rright | t^k =
    \frac{\sum_{\sigma \in \Tau^{(d+1)}} \sum_{\vect{w} \in \Pi_\sigma \rleft[ \xi \rright)
    \cap \rleft( \vect{v} + \Gamma \rright)} t^{\mathrm{ht}\rleft( \vw\rright)}}{(1-t)^{d+1}}\eqqcolon
    \frac{h^*_{P, \vect{v} + \Gamma} (t)}{(1-t)^{d+1}}
  \end{align*}
  where $\mathrm{ht} \colon \R^{d+1} \to \R;\vw = \rleft( w_{0}, \ldots, w_{d}
  \rright) \mapsto w_{0}$ is the projection onto the first coordinate.  In
  particular, for $\vect{v} = \vect{0}$ and $\Gamma = \Z^{d+1}$ this yields the
  usual equality $\sum_{k=0}^\infty \rleft| kP \cap \Z^d \rright| t^k = h^*_P( t
  ) / (1-t)^{d+1}$. We obtain
  \begin{align*}
    h^*_{P, \vect{v} + \Gamma} ( t ) = \sum_{k=0}^s
    h_{k,\vect{v}+\Gamma}^* t^k = \sum_{\sigma \in \Tau^{(d+1)}}
    \sum_{\vw \in \Pi_\sigma[\xi) \cap \rleft( \vv + \Gamma \rright)}
    t^{\hst{\Tau, \xi}{\vw}} \text{.}
  \end{align*}
  Analogous to Equation \eqref{eq:h-star-non-vanish}, it follows that
  \begin{align*}
    \hst{\Tau,\xi}{ C \cap \rleft( \vv +\Gamma \rright)} = 
    \set{ k \in \N  \with h_{k, \vv + \Gamma}^* \neq 0 } \text{,}
  \end{align*}
  so, in particular, this set is independent of the choice of $\rleft( \Tau_0,
  \xi_0 \rright) \in \Upsilon_C \times \Xi_C$.
\end{proof}

The following two propositions will be used in our proof of
Theorem~\ref{thm:nog-vG}.  We will prove them below in
Section~\ref{sec:key-observation} and Section~\ref{sec:fliplemma}, respectively.

\begin{prop}[Changing the generic vector]
  \label{prop:key-obs}
  Let $\Tau \in \Upsilon_C$ and $\vx \in C \cap \Z^{d+1}$. Then there exists
  $a,b \in \Z_{\ge0}$ with $a \le b$ such that
  \begin{align*}
    \set{ h^*_{\Tau, \xi} \rleft( \vect{x} \rright) \with \xi \in
    \Xi_C } = [ a, b ] \cap \Z \text{.}
  \end{align*}
\end{prop}

\begin{prop}[Changing the triangulation]
  \label{prop:fliplemma}
  Let $\vx \in C \cap \rleft(\vect{v} + \Gamma \rright)$, $\xi \in \Xi_C$ and
  $\Tau, \Tau' \in \Upsilon_C$ be two triangulations. Then there exist $\rleft(
  \Sm_1, \xi_1, \vect{y}_1 \rright), \ldots, \rleft( \Sm_R, \xi_R, \vect{y}_R
  \rright) \in \Upsilon_C \times \Xi_C \times \rleft( C \cap \rleft( \vect{v} +
  \Gamma \rright) \rright)$ such that
  \begin{align*}
    \set{ h^*_{\Sm_i,\xi_i} \rleft( \vect{y}_i \rright) \with
    i = 1, \ldots, R } \cup \set{ h^*_{\Tau,\xi} \rleft(
    \vect{x} \rright), h^*_{\Tau',\xi} \rleft( \vect{x} \rright) }
    = [ a, b ] \cap \Z \text{,}
  \end{align*}
  for two integers $a, b \in \Z_{\ge0}$ with $a \le b$.
\end{prop}

\begin{proof}[Proof of Theorem~\ref{thm:nog-vG}]
  Let $a_\vv\coloneqq \min \set{ \hst{\Tau,\xi}{ C \cap \rleft( \vv +\Gamma
      \rright)} }$ and $b_\vv \coloneqq \max \set{ \hst{\Tau,\xi}{ C \cap
      \rleft( \vv +\Gamma \rright)} }$.
	
  Note that we may replace $\vv$ by any element in $\vv + \Gamma$ without
  changing the statement. In particular, we may assume that $\vv \in C$ and that
  $h^*_{\Tau,\xi}(\vv) = a_{\vect{v}}$.  Moreover, there exists an element $\vw
  \in C \cap (\vv + \Gamma)$ with $h^*_{\Tau,\xi}(\vw) = b_{\vect{v}}$. We can
  write $\vw = \vv + \sum_{i=1}^r \vv_i - \sum_{i=1}^s \vv'_i$ with $\vv_1,
  \dotsc, \vv_r, \vv'_1, \dotsc, \vv'_s, \in (\set{1} \times P) \cap \Z^{d+1}$.
	
  Let $\vx_k \coloneqq \vv + \sum_{i=1}^k \vv_i$ for $0 \leq k \leq r$. We are
  going to show that for each $k = 1, \dotsc, r$, there are $\rleft( \Sm_{1},
  \xi_{1}, \vw_{1} \rright), \ldots, \rleft( \Sm_{q}, \xi_{q}, \vw_{q} \rright)
  \in \Upsilon_C \times \Xi_C \times \rleft( C \cap \rleft( \vv + \Gamma
  \rright) \rright)$ such that the corresponding $\hst{\Sm_j,\xi_j}{\vw_j}$ fill
  the gap between $\hst{\Tau,\xi}{\vx_{k-1}}$ and $\hst{\Tau,\xi}{\vx_{k}}$.
  Thus, we can fill the gap between $\vv$ and $\vv + \sum_{i=1}^r \vv_i$.  By
  symmetry, we then can also fill the gap between $\vw$ and $\vw + \sum_{i=1}^s
  \vv'_i = \vv + \sum_{i=1}^r \vv_i$, so the claim follows.
  
  By Proposition~\ref{prop:key-obs} and Proposition~\ref{prop:fliplemma}, it is
  in fact sufficient to show that there exists a triangulation $\Tau' \in
  \Upsilon_C$ and a generic vector $\xi' \in \Xi_C$ such that the gap between
  $\hst{\Tau', \xi'}{\vx_{k-1}}$ and $\hst{\Tau',\xi'}{\vx_{k}}$ can be filled.
	
  For this, let $\Tau'\in \Upsilon_C$ be a pulling triangulation (see, for
  instance, \cite[Section~4.3.2]{DRS:Triangulations}) which uses $\vv_k$ as its
  \emph{last} vertex. Then $\vv_k$ is an extremal ray generator of every
  full-dimensional cone in $\Tau'$,
  cf. \cite[Lemma~4.3.6~(2)]{DRS:Triangulations}.  Choose $\sigma \in \rleft(
  \Tau' \rright)^{(d+1)}$ such that $\vx_{k-1} \in \sigma$ and choose $\xi' \in
  \sigma^\circ \cap \Xi_C$.  Then $\vx_{k-1} \in \sigma\rleft[\xi'\rright)$ and
  it holds that
  \begin{align*}
    \hst{\Tau',\xi'}{\vx_k} = \hst{\Tau',\xi'}{\vx_{k-1} + \vv_k}
    = \hst{\Tau',\xi'}{\vx_{k-1}} \text{.}
  \end{align*}
  Thus, the claim follows. The precise way in which we apply
  Proposition~\ref{prop:key-obs} and Proposition~\ref{prop:fliplemma} is also
  indicated in Figure~\ref{fig:prop-nog}, where an arrow
  \enquote{$\leftrightarrow$} means that the gap between the two endpoints can
  be filled.
\end{proof}

  \begin{figure}[!ht]
    \centering
    \begin{tikzpicture}[description/.style={fill=white,inner sep=2pt}]
      \matrix (m) [matrix of math nodes, row sep=2.5em,
      column sep=2.5em, text height=1.5ex, text depth=0.25ex]
      {   \hst{\Tau ,\xi'}{\vx_{k-1}} &
        \hst{\Tau',\xi'}{\vx_{k-1}} &
        \hst{\Tau',\xi'}{\vx_k    } &
        \hst{\Tau ,\xi'}{\vx_k    } \\
        \hst{\Tau ,\xi }{\vx_{k-1}} &&&
        \hst{\Tau ,\xi }{\vx_k    } \text{.}\\
      };
      \draw [latex-latex] (m-1-1) -- (m-1-2)
      node[midway,above]{\tiny Prop.~\ref{prop:fliplemma}};
      \draw [=,double,double distance=1pt] (m-1-2) -- (m-1-3);
      \draw [latex-latex] (m-1-3) -- (m-1-4)
      node[midway,above]{\tiny Prop.~\ref{prop:fliplemma}};
      \draw [latex-latex] (m-1-1) -- (m-2-1)
      node[midway,right]{\tiny Prop.~\ref{prop:key-obs}};
      \draw [latex-latex] (m-1-4) -- (m-2-4)
      node[midway,left]{\tiny Prop.~\ref{prop:key-obs}};
      \draw [latex-latex,dashed] (m-2-1) -- (m-2-4);
    \end{tikzpicture}
    \caption{How to fill the gap between $h^*_{\Tau,\xi} \rleft( \vect{x}_{k-1}
      \rright)$ and $h^*_{\Tau,\xi} \rleft( \vect{x}_k \rright)$ in the proof of
      Theorem~\ref{thm:nog-vG}.}
    \label{fig:prop-nog}
  \end{figure}
  
\subsection{Changing the generic vector}
\label{sec:key-observation}
In this subsection, we are going to prove Proposition~\ref{prop:key-obs}.  The
next lemma is used in that proof.

\begin{lem}
  \label{lem:simpl-compl}
  Let $\Tau \in \Upsilon_C$ and $\sigma \in \Tau$. Then the set
  \begin{align*}
    \Lambda_{\Tau,\sigma} \coloneqq \set{ I_\xi \rleft( \sigma' \rright) \with
      \sigma' \in \Tau^{(d+1)}, \xi \in \Xi_C, \sigma^\circ \subseteq
      \sigma'\rleft[ \xi \rright)}
  \end{align*}
  is an abstract simplicial complex, \ie closed under taking subsets.
\end{lem}
\begin{proof}
  We show that if $S \in \Lambda_{\Tau,\sigma}$ and $\vect{v} \in S$, then
  $\rleft( S \setminus \{ \vect{v} \} \rright) \in \Lambda_{\Tau,\sigma}$ as
  well. Hence every subset of $S$, which can be achieved by repeatedly removing
  vectors from $S$, is contained in $\Lambda_{\Tau,\sigma}$.
	
  Take $S \in \Lambda_{\Tau,\sigma}$, choose $\sigma' \in \Tau^{(d+1)}$, $\xi
  \in \Xi_C$ such that $\sigma^\circ \subseteq \sigma'\rleft[ \xi \rright)$ and
  $S = I_\xi \rleft( \sigma' \rright)$, and let $\vect{v} \in S$. For $t \ge 0$
  let $\xi_t \coloneqq \xi + t \vect{v}$. Clearly $\xi_t \in C$ for all $t \ge
  0$. Also, as $\xi_0 = \xi$ is generic it follows that $\xi_t$ is generic for
  all but finitely many values of $t$. For a sufficiently large choice of $t$
  the coefficient of $\vect{v}$ in the linear combination $\xi_t =
  \sum_{\vect{w} \in \rleft( \sigma' \rright)^{(1)}} \mu_{\vect{w},t} \vect{w}$
  is positive and hence $I_{\xi_t} \rleft( \sigma' \rright) = S \setminus
  \set{\vect{v}}$ (see Figure~\ref{fig:xit}).
	
  Moreover, note that $\sigma^\circ \subseteq \sigma' \rleft[ \xi \rright)$ if
  and only if $I_\xi \rleft( \sigma' \rright) \subseteq \sigma^{(1)}$.  As
  $I_{\xi_t} \rleft( \sigma' \rright) \subseteq I_{\xi} \rleft( \sigma'
  \rright)$, it follows that $\sigma^\circ \subseteq \sigma' \rleft[ \xi_t
  \rright)$ and thus $S \setminus \set{\vect{v}} \in \Lambda_{\Tau, \sigma}$.
\end{proof}

  \begin{figure}[!ht]
    \centering
    \begin{tikzpicture}
      \draw (0,0) -- (2.5,2.5) node[right] {$C$};
      \draw (0,0) -- (-7.5,2.5);
      \draw[thin] (0,0) -- (0,2.5);
      \draw[thin] (0,0) -- (-2.5,2.5);
      \draw (-3,1) node[above] {$\sigma'$} -- (1,1) node[below right] {$\{ 1 \} \times P$};
      \fill (0.25,1.5) circle (1pt) node[above]{$\xi$};
      \fill (-2.75,2.5) circle (1pt) node[below]{$\xi_t$};
      \fill (-3,1) circle (1pt) node[below]{$\vect{v}$};
      
      \draw[dashed,-latex] (0.25,1.5) -- (-2.75,2.5);
    \end{tikzpicture}
    \caption{The point $\xi_t$ for large $t>0$.}
    \label{fig:xit}
  \end{figure}
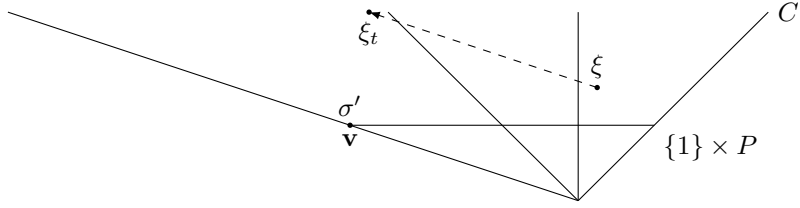

\begin{proof}[Proof of Proposition~\ref{prop:key-obs}]
  There exists a unique cone $\sigma \in \Tau$ such that $\vect{x} \in
  \sigma^\circ$. This cone does not need to be full-dimensional.  We can
  represent $\vect{x}$ as a linear combination $\vect{x} = \sum_{\vect{v} \in
    \sigma^{(1)}} \lambda_\vect{v} \vect{v}$ for positive real numbers
  $\lambda_\vect{v} > 0$. For a given $\xi \in \Xi_C$, there exists a unique
  full-dimensional cone $\sigma' \in \Tau^{(d+1)}$, such that $\sigma^\circ
  \subseteq \sigma'[\xi)$, and hence
  \begin{equation}\label{eq:key-obs}
    h^*_{\Tau, \xi}(\vect{x}) = \sum_{\vect{v} \in \sigma^{(1)}} \lFrac{\lambda_\vect{v}} +
    \rleft| I_\xi \rleft( \sigma' \rright) \cap \set{ \vect{v} \in
      \sigma^{(1)} \with \lambda_\vect{v} \in \Z  } \rright| \text{,}
  \end{equation}
  where $\lFrac{\lambda_\vect{v}}$ denotes the \emph{fractional part} of
  $\lambda_\vect{v}$, \ie $\lambda_\vect{v} - \lfloor \lambda_\vect{v} \rfloor$.
  By Lemma~\ref{lem:simpl-compl}, $\Lambda_{\Tau, \sigma, \vect{x}} \coloneqq
  \set{ S \cap \set{ \vect{v} \in \sigma^{(1)} \with \lambda_{\vect{v}} \in \Z }
    \with S \in \Lambda_{\Tau,\sigma} }$ is an abstract simplicial complex (a
  subcomplex of $\Lambda_{\Tau,\sigma}$). It follows from Eq.~\eqref{eq:key-obs}
  that
  \begin{align*}
    \set{ h^*_{\Tau, \xi} \rleft( \vect{x} \rright) \with \xi \in
    \Xi_C } \subseteq \sum_{\vect{v} \in \sigma^{(1)}}
    \lFrac{\lambda_{\vect{v}}} + \set{ \rleft| S' \rright| \with S'
    \in \Lambda_{\Tau, \sigma, \vect{x}} } \text{.}
  \end{align*}
  The other inclusion \enquote{$\supseteq$} follows by the fact that every $S'
  \in \Lambda_{\Tau, \sigma, \vect{x}}$ has a presentation $S' = I_\xi \rleft(
  \sigma' \rright)$ for $\xi \in \Xi_C$ and $\sigma' \in \Tau^{(d+1)}$ with
  $\sigma^\circ \subseteq \sigma'[\xi)$ (see Lemma~\ref{lem:simpl-compl}).
  Hence
  \begin{align*}
    \set{ h^*_{\Tau, \xi} \rleft( \vect{x} \rright) \with \xi \in
    \Xi_C } = [ a, a + b ] \cap \Z \text{,}
  \end{align*}
  with $a = \sum_{\vect{v} \in \sigma^{(1)}} \lFrac{\lambda_\vect{v}}$ and $b =
  \dim \Lambda_{\Tau, \sigma, \vect{x}} + 1$ where the dimension of an abstract
  simplicial complex is the largest dimension of any of its faces $S$ which in
  turn is $\dim S = | S | - 1$.
\end{proof}

\begin{ex}
  \label{ex:change-triang}
  If we let $\Tau \in \Upsilon_C$ also vary in Proposition~\ref{prop:key-obs},
  then the analogous statement is false in general.
	
  Denote the standard basis of $\R^6$ by $\vect{e}_1, \ldots, \vect{e}_6$ and
  consider the lattice polytope
  \begin{align*}
    P \coloneqq \conv \rleft( 5 \vect{e}_1 - 4 \rleft( \vect{e}_2 + \vect{e}_3 +
    \vect{e}_4 \rright) - 3 \rleft( \vect{e}_5 + \vect{e}_6 \rright),
    \vect{e}_2, \ldots, \vect{e}_6, \vect{0}, 5 \vect{e}_1 - \vect{e}_2 -
    \ldots - \vect{e}_6 \rright) \text{,}
  \end{align*}
  whose only lattice points are its vertices (such polytopes are called
  \emph{empty}). We denote the vertices of $\{ 1 \} \times P \subseteq \R^7$ by
  $\vect{v}_i$ for $i = 1, \ldots, 8$ where the order is taken to be the one as
  they appear in the definition above. Let $C \subseteq \R^7$ be the cone over
  $P$. As $P$ is a circuit (see Remark~\ref{rem:corank1-ptc}), $\Upsilon_C$
  consists of two triangulations $\Tau_+, \Tau_-$ where
  \begin{align*}
    \Tau_+ &= \set{ \cone \rleft( \vect{v}_1, \ldots,
             \vect{v}_{i-1}, \vect{v}_{i+1}, \ldots, \vect{v}_{8} \rright)
             \with i = 1, \ldots, 6 } \text{,}\\
    \Tau_- &= \set{ \cone \rleft( \vect{v}_1, \ldots,
             \vect{v}_{i-1}, \vect{v}_{i+1}, \ldots, \vect{v}_{8} \rright)
             \with i = 7, 8 } \text{.}
  \end{align*}
  We take the lattice point $\vect{x} \coloneqq 4 \vect{e}_0 + \vect{e}_1$ in
  $C$ which has representations
  \begin{align*}
    \vect{x} = \tfrac{1}{5} \vv_1 + \tfrac{4}{5} \rleft( \vv_2 + \vv_3 + \vv_4 \rright)
    + \tfrac{3}{5} \rleft( \vv_5 + \vv_6 \rright) + \tfrac{1}{5} \vv_7 =
    \tfrac{1}{5} \rleft( \vv_2 + \ldots + \vv_6 + \vv_8 \rright) +
    \tfrac{14}{5} \vv_7 \text{.}
  \end{align*}
  For every $\xi \in \Xi_C$, we obtain $h^*_{\Tau_-,\xi} \rleft( \vect{x}
  \rright) = 4$ while $h^*_{\Tau_+,\xi} \rleft( \vect{x} \rright) = 2$, and thus
  $h^*_{\Upsilon_C,\Xi_C} \rleft( \vect{x} \rright) = \{ 2, 4 \}$ is missing the
  number $3$. On the other hand if we fix $\Tau \in \Upsilon_C$, then $\rleft|
  h^*_{\Tau,\Xi_C} \rleft( \vect{x} \rright) \rright| = 1$ and hence does not
  has a gap.
\end{ex}

\subsection{Changing the triangulation}
\label{sec:fliplemma}

The proof of Proposition~\ref{prop:fliplemma} relies on flips in triangulations
and the fact that any two regular triangulations can be connected by a sequence
of flips. We recall some notions and results and refer to
\cite{DRS:Triangulations} (see also \cite[Chapter~7, Section~2]{GKZ}) for
details and references.

\begin{defn}
  \label{def:ptc}
  A \emph{(homogeneous) vector set} in $\R^{d+1}$ is a finite subset $\Am
  \subseteq \R^{d+1}$, such that the first component of each $\vv \in \Am$ is
  $1$.  The number $|\Am| - \dim( \lspan(\Am) )$ is called its \emph{corank}. We
  will never consider inhomogeneous vector sets, hence we will omit the
  specifier \enquote{homogeneous}.
	
  In this paper under a \emph{polyhedral subdivision} $\Sm$ of $\Am$ we will
  understand a subset $\Sm$ of the power set of $\Am$ such that
  \begin{enumerate}
  \item $\set{ \cone( B ) \with B \in \Sm }$ forms a polyhedral subdivision of
    the cone generated by $\Am$, \ie $C_\Am \coloneqq \cone(\Am)$, and
  \item for every $B, B' \in \Sm$
    \begin{align*}
      B \cap \rleft( \cone(B) \cap \cone(B') \rright) = B' \cap \rleft( \cone(B)
      \cap \cone(B') \rright) \text{.}
    \end{align*}
  \end{enumerate}
  A \emph{cell} $B \in \Sm$ is called \emph{simplicial} if it consists of
  linearly independent vectors. A \emph{triangulation} $\Tau$ of $\Am$ is a
  polyhedral subdivision such that all its cells are simplicial.
\end{defn}

\begin{rem}
  \label{rem:special-pt-conf}
  The vector sets which we will deal with in this paper come from lattice points
  on height $1$ contained in the cone over lattice polytopes. In particular,
  subtleties in connection with \enquote{double points} won't appear.

  Given a simplicial cell $B$ of a polyhedral subdivision $\Sm$ of a vector set
  $\Am$, the set $B$ necessarily consists of the primitive generators of the
  extremal rays of $\cone(B)$. In particular, $B$ and $\cone(B)$ uniquely
  determine each other. Hence there is a natural correspondence between
  triangulations of cone $C_\Am$ as defined in Section~\ref{subsec:hot} and
  triangulations of the vector set $\Am$. However, for an arbitrary cell $B$ in
  a polyhedral subdivision $\Sm$ of $\Am$, it is necessary to remember $B$, as
  $\cone(B)$ does not determine $B$ in general. One might want to think of $B$
  as the \enquote{markings} of $\cone(B)$.
\end{rem}

A \emph{refinement} $\Sm'$ of a polyhedral subdivision $\Sm$ is a polyhedral
subdivision where for each $B' \in \Sm'$ there exists $B \in \Sm$ such that $B'
\subseteq B$.

An \emph{almost-triangulation} of a vector set $\Am$ is a pair $\rleft( \Bm, \Sm
\rright)$ of a subset $\Bm \subseteq \Am$ and a polyhedral subdivision $\Sm$ of
simultaneously both $\Am$ and $\Bm$ such that it is not a triangulation but all
its proper refinements (with respect to $\Bm$) are one.

\begin{prop}[{see \cite[Corollary~2.4.6]{DRS:Triangulations}}]
  Every almost-triangulation has exactly \emph{two} proper refinements, which
  are both triangulations.
\end{prop}

Two triangulations $\Tau_1$, $\Tau_2$ of the same vector set $\Am$ are
\emph{connected by a flip} if there is an almost-triangulation $\rleft( \Bm, \Sm
\rright)$ of $\Am$ such that $\Tau_1$ and $\Tau_2$ are the only two
triangulations refining $\rleft( \Bm, \Sm \rright)$.

\begin{ex}
  \label{ex:flip}
  Consider the vector set $\Am \coloneqq \set{ \vv_1, \ldots, \vv_{10} }
  \subseteq \R^3$ whose projection to $\R^2$ by forgetting the first coordinate
  is given in Figure~\ref{fig:flip}. To simplify notation we will abbreviate the
  subset $\{ \vv_{i_1}, \ldots, \vv_{i_k} \} \subseteq \{ \vv_1, \ldots,
  \vv_{10} \}$ by \enquote{$i_1 \ldots i_k$}.
  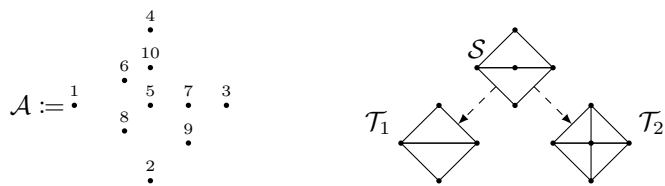
\begin{figure}[!ht]
    \centering
    \begin{tikzpicture}
      \fill (1,2) node[above]{\tiny $1$} circle (1pt);
      \fill (3,2) node[above]{\tiny $3$} circle (1pt);
      \fill (2,1) node[above]{\tiny $2$} circle (1pt);
      \fill (2,2) node[above]{\tiny $5$} circle (1pt);
      \fill (2,3) node[above]{\tiny $4$} circle (1pt);
      \fill (1.66,2.33) node[above]{\tiny $6$} circle (1pt);
      \fill (2.5,2) node[above]{\tiny $7$} circle (1pt);
      \fill (1.66,1.66) node[above]{\tiny $8$} circle (1pt);
      \fill (2.5,1.5) node[above]{\tiny $9$} circle (1pt);
      \fill (2,2.5) node[above]{\tiny $10$} circle (1pt);
      \node[left] at (1,2) {$\Am\coloneqq$};
    \end{tikzpicture}\hspace{4em}
    \begin{tikzpicture}[scale=0.5]
      \draw (-1,0) -- (1,0) -- (0,1) -- cycle;
      \draw (-1,0) -- (1,0) -- (0,-1) -- cycle;
      \fill (-1,0) node[above left]{$\Tau_1$} circle (2pt);
      \fill (1,0) circle (2pt);
      \fill (0,1) circle (2pt);
      \fill(0,-1) circle (2pt);
      
      \draw (1,2) -- (3,2) -- (2,1) -- cycle;
      \draw (1,2) -- (3,2) -- (2,3) -- cycle;
      \fill (1,2) node[above]{$\Sm$} circle (2pt);
      \fill (3,2) circle (2pt);
      \fill (2,1) circle (2pt);
      \fill (2,2) circle (2pt);
      \fill (2,3) circle (2pt);

      \draw (3,0) -- (5,0) -- (4,-1) -- cycle;
      \draw (3,0) -- (5,0) -- (4,1) -- cycle;
      \draw (4,-1) -- (4,1);
      \fill (3,0) circle (2pt);
      \fill (5,0) node[above right]{$\Tau_2$} circle (2pt);
      \fill (4,0) circle (2pt);
      \fill (4,-1) circle (2pt);
      \fill (4,1) circle (2pt);

      \draw[-latex,dashed] (1.5,1.5) -- (0.5,0.5);
      \draw[-latex,dashed] (2.5, 1.5) -- (3.5, 0.5);
    \end{tikzpicture}
    \caption{The flip of Example~\ref{ex:flip}.}
    \label{fig:flip}
  \end{figure}
  The two triangulations $\Tau_1 \coloneqq \{ 123,\allowbreak 134,\allowbreak
  12,\allowbreak 23,\allowbreak 13,\allowbreak 14,\allowbreak 34,\allowbreak
  1,\allowbreak 2,\allowbreak 3,\allowbreak 4, \allowbreak \emptyset \}$ and
  $\Tau_2 \coloneqq \{ 125,\allowbreak 235,\allowbreak 345,\allowbreak
  145,\allowbreak 12,\allowbreak 15,\allowbreak 25,\allowbreak 35,\allowbreak
  23,\allowbreak 45,\allowbreak 34,\allowbreak 14,\allowbreak 1,\allowbreak
  2,\allowbreak 3,\allowbreak 4,\allowbreak 5, \allowbreak \emptyset \}$ are
  connected by a flip supported on the almost-triangulation $\rleft( \Bm, \Sm
  \rright)$ where $\Bm \coloneqq \{ 1, \ldots, 5 \}$ and $\Sm = \{
  1235,\allowbreak 1345,\allowbreak 153,\allowbreak 12,\allowbreak
  23,\allowbreak 14,\allowbreak 34,\allowbreak 1,\allowbreak 2,\allowbreak
  3,\allowbreak 4, \allowbreak \emptyset \}$.
\end{ex}

Vector sets of corank $1$ will play an important role in the proof, so let us
recall some facts. We refer to \cite[Section~2.4]{DRS:Triangulations} for
details.

\begin{rem}
  \label{rem:corank1-ptc}
  A corank $1$ vector set $\Am \subseteq \R^{d+1}$ possesses a unique linear
  dependence relation, say $\vect{0} = \sum_{\vv \in \Am} \lambda_{\vv} \vv$,
  which partitions $\Am$ into three subsets $\Am_-$, $\Am_0$ and $\Am_+$:
  \begin{align*}
    \Am_+ \coloneqq \set{ \vv \in \Am \with \lambda_{\vv} > 0 }, \quad
    \Am_0 \coloneqq \set{ \vv \in \Am \with \lambda_{\vv} = 0 }, \quad
    \Am_- \coloneqq \set{ \vv \in \Am \with \lambda_{\vv} < 0 } \text{.}
  \end{align*}
  The following are the only two triangulations of $\Am$, both are regular.
  \begin{align*}
    \Tau_+ \coloneqq \set{ \cone \rleft( B \rright) \colon \Am_+ \not\subseteq
    B \subseteq \Am } \qquad \text{and} \qquad
    \Tau_- \coloneqq \set{ \cone \rleft( B \rright) \colon \Am_- \not\subseteq
    B \subseteq \Am } \text{.}
  \end{align*}
  By \cite[Theorem~4.4.1]{DRS:Triangulations}, the two triangulations $\Tau_+$
  and $\Tau_-$ form the prototype of a flip. A lattice polytope $P \subseteq
  \R^d$ such that its homogenized vertices form a vector set $\Am$ of corank $1$
  with $\Am_0 = \emptyset$ is called a \emph{circuit}.
\end{rem}

Proposition~\ref{prop:fliplemma} will follow from the following further
reduction to the case of corank $1$.
\begin{lem}
  \label{lem:fliplemma-crk-1}
  Let $\Am \subseteq \Z^{d+1}$ be a vector set of corank $1$, $\vv \in
  \Z^{d+1}$, $\xi \in \Xi_{C_\Am}$, $\Gamma_\Am$ be the sublattice of $\Z^{d+1}$
  generated by $\Am$ and $\vx \in C_\Am \cap (\vv + \Gamma_\Am)$. Then there
  exist $\rleft( \Sm_1, \xi_1, \vect{y}_1 \rright), \ldots, \rleft( \Sm_R,
  \xi_R, \vect{y}_R \rright) \in \set{ \Tau_+, \Tau_- } \times \Xi_{C_\Am}
  \times \rleft( C_\Am \cap \rleft( \vect{v} + \Gamma_\Am \rright) \rright)$
  such that
  \begin{align*}
    \set{ h^*_{\Sm_i, \xi_i} \rleft( \vect{y}_i \rright) \with
    i = 1, \ldots, R } \cup \set{ h^*_{\Tau_+,\xi}
    \rleft( \vect{x} \rright), h^*_{\Tau_-,\xi} \rleft( \vect{x} \rright) }
    = \rleft[ a, b \rright] \cap \Z \text{,}
  \end{align*}
  for two nonnegative integers $a \le b$.
\end{lem}

We will prove Lemma~\ref{lem:fliplemma-crk-1} in Section~\ref{sec:corank1}. The
following technical lemma will be needed to make a generic point \enquote{more}
generic.

\begin{lem}
  \label{lem:tweak-gp}
  Let $C \subseteq \R^{d+1}$ be a full-dimensional cone, $\sigma \subseteq C$ a
  simplicial full-dimensional subcone and $H_1, \ldots, H_R \subseteq \R^{d+1}$
  a family of (linear) hyperplanes.  For every $\xi \in C$, there exists $\xi'
  \in C^\circ \setminus \rleft( \bigcup_{j=1}^R H_j \rright)$ with $I_\xi
  \rleft( \sigma \rright) = I_{\xi'} \rleft( \sigma \rright)$ (see
  Definition~\ref{def:hot}).
\end{lem}
\begin{proof}
  As $\sigma$ is full-dimensional and simplicial, there exists a unique
  representation $\xi = \sum_{\vect{v} \in \sigma^{(1)}} \lambda_{\vect{v}}
  \vect{v}$. Further, there exists $\vx \in \sigma^\circ \setminus \rleft(
  \bigcup_{j=1}^R H_j \rright)$, which has a representation $\vx = \sum_{\vv \in
    \sigma^{(1)}} \mu_{\vv} \vect{v}$ with $\mu_\vv > 0$ for all $\vv \in
  \sigma^{(1)}$.
	
  For $0 \leq t \leq 1$ let $\xi(t) \coloneqq (1-t) \xi + t \vx$. Clearly
  $\xi(t) \in C^\circ$ for all $0 < t \leq 1$. As $\xi(1) \notin \bigcup_{j=1}^R
  H_j$, the points $\xi(t)$ avoid the hyperplanes $H_i$ for all but finitely
  many values of $t$.  If we choose $t$ close to $0$, then $\mu_\vv > 0$ for
  every $\vv \in \sigma^{(1)}$ implies that the nonzero coefficients of $\xi(t)$
  in the basis $\sigma^{(1)}$ have the same signs as the $\lambda_{\vect{v}}$,
  and the zero coefficients become positive. Hence for such a choice of $t$,
  $\xi' \coloneqq \xi(t)$ satisfies the claim.
\end{proof}

\begin{proof}[Proof of Proposition~\ref{prop:fliplemma}]
  As both $\Tau$ and $\Tau'$ are regular triangulations of the same vector set
  $\Am_0 \coloneqq \rleft( \{ 1 \} \times P \rright) \cap \Z^{d+1}$, they are
  connected by a sequence of flips (see, for instance,
  \cite[Theorem~5.3.7]{DRS:Triangulations}), \ie there is a finite sequence of
  triangulations of the vector set $\Am_0$ such that every two consecutive
  triangulations differ in a flip.  It is sufficient to prove our claim for
  every pair of consecutive triangulations in this sequence, and hence we assume
  from now on that $\Tau$ and $\Tau'$ differ only by a flip.
	
  Let $\rleft( \Bm, \Sm \rright)$ be the almost-triangulation such that $\Tau$
  and $\Tau'$ are the two proper refinements of it. According to
  \cite[Lemma~2.4.5]{DRS:Triangulations}, each cell $\Am \in \Sm^{(d+1)}$ has
  corank at most $1$. Take $\Am \in \Sm^{(d+1)}$ such that $\vect{x} \in \sigma
  \coloneqq \cone \rleft( \Am \rright)$ and fix $\xi' \in \Xi_C \cap
  \sigma^\circ$.

  If $\Am$ has corank $0$, then $\Am \in \Tau \cap \Tau'$, and thus
  $h^*_{\Tau,\xi'} \rleft( \vect{x} \rright) = h^*_{\Tau',\xi'} \rleft( \vect{x}
  \rright)$. The statement follows by Proposition~\ref{prop:key-obs}.

  If $\Am$ has corank $1$, then (up to swapping $\Tau_+$ and $\Tau_-$) we may
  assume $\Tau_+ \subseteq \Tau$ and $\Tau_- \subseteq \Tau'$ where $\Tau_\pm$
  denote the two triangulations of $\Am$ (see Remark~\ref{rem:corank1-ptc}). We
  obtain
  \begin{align*}
    h^*_{\Tau,\xi'} \rleft( \vect{x} \rright) =
    h^*_{\Tau_+,\xi'} \rleft( \vect{x} \rright) \qquad\text{and}\qquad 
    h^*_{\Tau',\xi'} \rleft( \vect{x} \rright) =
    h^*_{\Tau_-,\xi'} \rleft( \vect{x} \rright) \text{.}
  \end{align*}
  By Lemma~\ref{lem:fliplemma-crk-1}, there exist $\rleft( \Sm_1, \xi_1,
  \vect{y}_1 \rright), \ldots, \rleft( \Sm_R, \xi_R, \vect{y}_R \rright) \in \{
  \Tau_+,\Tau_- \} \times \Xi_{C_{\Am}} \times \rleft( C_{\Am} \cap \rleft(
  \vect{v} + \Gamma_\Am \rright) \rright)$ such that the $h^*_{\Sm_i,\xi_i}
  \rleft( \vect{y}_i \rright)$ fill the gap between $h^*_{\Tau_+,\xi'} \rleft(
  \vect{x} \rright)$ and $h^*_{\Tau_-,\xi'} \rleft( \vect{x} \rright)$. By
  Lemma~\ref{lem:tweak-gp}, we may assume $\xi_i \in \Xi_C$ without changing the
  value of $h^*_{\Sm_i,\xi_i} \rleft( \vect{y}_i \rright)$. Moreover $C_\Am \cap
  \rleft( \vv + \Gamma_\Am \rright) \subseteq C \cap \rleft( \vv + \Gamma
  \rright)$. If $\Sm_i = \Tau_+$ for some $i = 1, \ldots, R$, then
  $h^*_{\Sm_i,\xi_i} \rleft( \vect{y}_i \rright) = h^*_{\Tau,\xi_i} \rleft(
  \vect{y}_i \rright)$ (and analogously for $\Sm_i = \Tau_-$). In particular, we
  can fill the gap between $h^*_{\Tau,\xi'}(\vect{x})$ and
  $h^*_{\Tau',\xi'}(\vect{x})$. The statement follows by
  Proposition~\ref{prop:key-obs}.
\end{proof}

\subsection{The corank \texorpdfstring{$1$}{1} case}
\label{sec:corank1}

In this section we will prove Lemma~\ref{lem:fliplemma-crk-1}. In its proof we
will consider certain functions which we want to discuss separately here.  We
denote by $\lFrac{x}$ the \emph{fractional part} of a real number $x$, \ie $x -
\floor{x}$ where $\floor{x}$ is the largest integer less than or equal to
$x$. For two finite families of positive integers $\rleft( \lambda_i \rright)_{i
  \in I}$ and $\rleft( \mu_j \rright)_{j \in J}$ with $\gcd\rleft( \lambda_i,
\mu_j \colon i \in I, j \in J \rright) = 1$ and $\sum_{i \in I} \lambda_i =
\sum_{j \in J} \mu_j$ and a further family $\rleft( x_k \rright)_{k \in I \cup
  J}$ of rational numbers, we define
\begin{align*}
  f \colon \R \to \Q; t \mapsto \sum_{i \in I} \lFrac{x_i - \lambda_i t} +
  \sum_{j \in J} \lFrac{x_j + \mu_j t} \text{,}
\end{align*}
which is a periodic bounded step function with period $1$. Such functions have
already appeared in number theory and algebraic geometry (see, for instance,
\cite{Vasyunin,Borisov:QuotSing,BB:BoundStep,Bober:FracRat}).

The function $f$ is piecewise constant and the interesting $t$-values are the
ones where $f(t)$ is different from its left-handed or right-handed limit. We
call those $t$ \emph{potential jump discontinuities} and observe that this is
the case if and only if $x_i - \lambda_i t \in \Z$ for some $i \in I$ or $x_j +
\mu_j t \in \Z$ for some $j \in J$. We define for a potential jump discontinuity
$t$
\begin{align*}
  l(t) \coloneqq \rleft| \set{ j \in J \with x_j + \mu_j t \in \Z } \rright|
  \qquad\text{and}\qquad
  r(t) \coloneqq \rleft| \set{ i \in I \with x_i - \lambda_i t \in \Z } \rright| \text{.}
\end{align*}
In the following $\lim_{t \to t_0-} f(t)$ (resp.~$\lim_{t \to t_0+} f(t)$) will
denote the left-handed (resp.~right-handed) limit of a function $f \colon \R \to
\R$.
\begin{lem}
  For a potential jump discontinuity $t_0 \in \R$ the relationship between
  $f(t_0)$ and its left- resp.~right-handed limit is given as follows:
  \begin{align*}
    \lim_{t \to t_0 -} f(t) = f \rleft( t_0 \rright) + l \rleft( t_0 \rright) \text{,}
    \qquad \text{and} \qquad
    \lim_{t \to t_0 +} f(t) = f \rleft( t_0 \rright) + r \rleft( t_0 \rright) \text{.}
  \end{align*}
  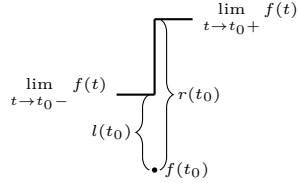
\begin{figure}[!ht]
    \centering
    \begin{tikzpicture}
      \draw[thick] (-0.5,0) node[left]{\tiny$\displaystyle\lim_{t\to t_0-}f(t)$} -- (0,0);
      \draw[thick] (0,0) -- (0,1);
      \draw[thick] (0,1) -- (0.5,1) node[right] {\tiny$\displaystyle\lim_{t\to t_0+}f(t)$};

      \draw [decorate,decoration={brace,amplitude=5pt},xshift=-2pt,yshift=0pt]
      (0,-1) -- (0,0) node[midway,left,xshift=-3pt]{\tiny$l \rleft( t_0
        \rright)$};

      \draw [decorate,decoration={brace,amplitude=5pt},xshift=2pt,yshift=0pt]
      (0,1) -- (0,-1) node[midway,right,xshift=3pt]{\tiny$r \rleft( t_0
        \rright)$};

      \fill (0,-1) circle (1pt) node[right]{\tiny$f\rleft( t_0 \rright)$};
    \end{tikzpicture}
    \caption{Relationship between $f(t_0)$ and its left- resp.~right-handed
      limit at a jump discontinuity $t_0$.}
    \label{fig:jd}
  \end{figure}
\end{lem}
\begin{proof}
  As $\sum_{i \in I} \lambda_i = \sum_{j \in J} \mu_j$, we can rewrite $f$ as
  follows
  \begin{align*}
    f \colon \R \to \Q; t \mapsto \sum_{i \in I \cup J} x_i - \sum_{i \in I} \floor{x_i + \lambda_i t} -
    \sum_{j \in J} \floor{x_j - \mu_j t} \text{.}
  \end{align*}
  The statement follows by the following properties of the floor-function. Let
  $x,t_0 \in \R$ and $\lambda, \mu \in \Z_{>0}$ with $x + \lambda t_0, x - \mu
  t_0 \in \Z$. Then

  \begin{align*}
    \floor{x + \lambda t_0} &= \lim_{t\to t_0-} \floor{x + \lambda t} + 1\text{,} 
    &\floor{x+\lambda t_0} &= \lim_{t\to t_0+} \floor{x + \lambda t} \text{,}\\
    \floor{x - \mu t_0} &= \lim_{t\to t_0-} \floor{x - \mu t} \text{,}
    &\floor{x - \mu t_0} &= \lim_{t\to t_0+} \floor{x - \mu t} +1\text{.}
  \end{align*}
\end{proof}

\begin{proof}[Proof of Lemma~\ref{lem:fliplemma-crk-1}]
  The coefficients in the linear dependence relation $\sum_{\vv \in \Am_+}
  \lambda_\vv \vv = \sum_{\vv \in \Am_-} \mu_\vv \vv$ can be chosen to be
  positive integers with $\gcd \rleft( \mu_{\vv_-}, \lambda_{\vv_+} \colon \vv_-
  \in \Am_-, \vv_+ \in \Am_+ \rright) = 1$. Let $\sigma' \in \Tau_-$ be the
  unique cone such that $\vect{x} \in \sigma' \rleft[ \xi \rright)$ and denote
  by $\vv' \in \Am_-$ the unique element of $\Am_- \setminus \rleft( \sigma'
  \rright)^{(1)}$. We can represent $\vect{x}$ as a linear combination $\vect{x}
  = \sum_{\vv \in \Am \setminus \set{ \vv'}} x_\vv \vv$ for nonnegative rational
  numbers $x_\vv$. Moreover we set $x_{\vv'} \coloneqq 0$.

  Take $\vv'' \in \Am_+$ such that $\tfrac{x_{\vv''}}{\lambda_{\vv''}} = \min
  \set{ \frac{x_\vv}{\lambda_\vv} \colon \vv \in \Am_+}$, so that $x_\vv -
  \frac{x_{\vv''}}{\lambda_{\vv''}}\lambda_\vv \ge 0$ for all $\vv \in
  \Am_+$. Let $\sigma'' \in \Tau_+$ be the unique cone such that $\vv''$ does
  not generate a ray of $\sigma''$.  We use the dependence relation to change
  the representation of $\vx$ to
  \begin{align*}
    \vect{x} = \sum_{\vv \in \Am_+} \rleft( x_\vv - t \lambda_\vv \rright) \vv +
    \sum_{\vv \in \Am_0} x_\vv \vv + 
    \sum_{\vv \in \Am_-} \rleft( x_\vv + t \mu_\vv \rright) \vv \text{.}
  \end{align*}
  We let $t \in \rleft[ 0, \tfrac{x_{\vv''}}{\lambda_{\vv''}} \rright]$, so that
  the coefficients in all representations of $\vect{x}$ are nonnegative. Further
  we consider the following periodic bounded step function with period $1$:
  \begin{align*}
    f \colon \R \to \Z; t \mapsto
    \sum_{\vv \in \Am_+} \lFrac{x_\vv - t \lambda_\vv} + 
    \sum_{\vv \in \Am_0} \lFrac{x_\vv} +
    \sum_{\vv \in \Am_-} \lFrac{x_\vv + t \mu_\vv} \text{.}
  \end{align*}
  Observe that $f$ takes integer values, as $\sum_{\vv \in \Am} x_\vv \in \Z$.
  Then $f(0) = h^*_{\Tau_-,\xi'} \rleft( \vect{x} \rright)$ and $f\rleft(
  \tfrac{x_{\vv''}}{\lambda_{\vv''}} \rright) = h^*_{\Tau_+,\xi''} \rleft(
  \vect{x} \rright)$ for $\xi' \in \Xi_{C_\Am} \cap \rleft( \sigma'
  \rright)^\circ$ and $\xi'' \in \Xi_{C_\Am} \cap
  \rleft(\sigma''\rright)^\circ$. The gap between $h^*_{\Tau_-,\xi}(\vect{x})$
  and $h^*_{\Tau_-,\xi'}(\vect{x})$ (resp.~the gap between
  $h^*_{\Tau_+,\xi}(\vect{x})$ and $h^*_{\Tau_+,\xi''}(\vect{x})$) can be filled
  by using Proposition~\ref{prop:key-obs}, so it remains to show that the gap
  between $h^*_{\Tau_-,\xi'}(\vect{x})$ and $h^*_{\Tau_+,\xi''}(\vect{x})$ can
  be also filled.

  Let $D$ be the set of potential jump discontinuities of $f$ which lie in the
  interval $\rleft[ 0, \tfrac{x_{\vv''}}{\lambda_{\vv''}} \rright]$.  Let $t_1 <
  t_2$ be two successive potential jump discontinuities. Then, if $f(t_1) <
  f(t_2)$ it also holds that $f(t_1) + r(t_1) \geq f(t_2)$, see
  Figure~\ref{fig:succ-jd}. Similarly, if $f(t_1) > f(t_2)$, then $f(t_1) \leq
  f(t_2) + l(t_2)$.  To finish our proof it is therefore sufficient to prove the
  following two claims:
  
  For each $t \in D \setminus\set{\frac{x_{\vv''}}{\lambda_{\vv''}}}$ with $r(t)
  > 0$, we claim that
  \[
    [f(t), f(t) + r(t) - 1] \cap \Z \subseteq \set{
      \hst{\Tau_+,\widetilde{\xi}}{\vy} \with \widetilde{\xi} \in \Xi_{C_\Am},
      \vy \in C_\Am \cap \rleft( \vv + \Gamma_\Am \rright) }\text{,}
  \]
  and similarly for $t \in D \setminus\set{0}$ with $l(t) > 0$, we claim that
  \[
    [f(t), f(t) + l(t) - 1] \cap \Z \subseteq \set{
      \hst{\Tau_-,\widetilde{\xi}}{\vy} \with \widetilde{\xi} \in \Xi_{C_\Am},
      \vy \in C_\Am \cap \rleft( \vv + \Gamma_\Am \rright) }\text{.}
  \]
  We only show the first claim, as the proof of the second one is analogous.
  For this, fix $t \in D\setminus\set{\frac{x_{\vv''}}{\lambda_{\vv''}}}$ with
  $r(t) > 0$.  Choose $\vv_0 \in \Am_+$ with $x_{\vv_0} - t \lambda_{\vv_0} \in
  \Z$ and let $\sigma_0 \coloneqq \cone \rleft( \Am \setminus \set{\vv_0}
  \rright) \in \Tau_+$ be the unique cone such that $\vv_0$ does not generate a
  ray of $\sigma_0$.  Let $\vy \coloneqq \vx - (x_{\vv_0} - t
  \lambda_{\vv_0})\vv_0 + \sum_{\vv \in \Am_+ \setminus \set{\vv_0 }}\vv$.  For
  $\xi_0 \in \sigma_0^\circ \cap \Xi_{C_\Am}$ it follows from
  Eq.~\eqref{eq:key-obs} that $\hst{\Tau_+, \xi_0}{\vy} = f(t)$.  On the other
  hand, it holds that $\vv_0 = 1/\lambda_{\vv_0}\rleft(\sum_{\vv \in \Am_-}
  \mu_\vv \vv - \sum_{\vv \in \Am_+ \setminus \set{\vv_0 }} \lambda_\vv
  \vv\rright)$ and thus $I_{\vv_0}(\sigma_0) = \Am_+ \setminus \set{\vv_0}$.  By
  Lemma~\ref{lem:tweak-gp}, we can find an element $\xi_1 \in \Xi_{C_\Am}$ with
  $I_{\xi_1}(\sigma_0) = I_{\vv_0}(\sigma_0)$. Using Eq.~\eqref{eq:key-obs}
  again, it follows that $\hst{\Tau_+, \xi_1}{\vy} = f(t) + r(t) - 1$. Finally,
  the gap between $f(t)$ and $f(t) + r(t) - 1$ can be filled by using
  Proposition~\ref{prop:key-obs}.

  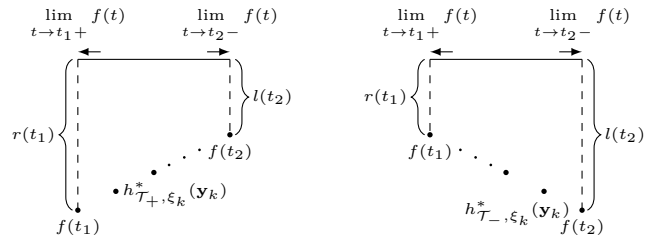
\begin{figure}[!ht]
    \centering
    \begin{tikzpicture}
      \draw (0,2) -- (2,2);
      
      \draw[-latex] (0.3,2.1) -- (0,2.1) node[above]{\tiny$\displaystyle \lim_{t\to t_1 +} f(t)$};
      \draw[-latex] (1.7,2.1) -- (2,2.1) node[above]{\tiny$\displaystyle \lim_{t\to t_2-}f(t)$};
      
      \draw[dashed] (0,2) -- (0,0);
      \draw[dashed] (2,2) -- (2,1);
      \fill (0,0) circle (1pt) node[below]{\tiny$f \rleft( t_1 \rright)$};
      \fill (2,1) circle (1pt) node[below]{\tiny$f \rleft( t_2 \rright)$};

      \fill (0.5,0.25) circle (1pt);
      \fill (1,0.5) circle (1pt) node[below,xshift=8pt]{\tiny$h^*_{\Tau_+,\xi_k}
        \rleft( \vect{y}_k \rright)$};
      \fill (1.2,0.6) circle (0.5pt);
      \fill (1.4,0.7) circle (0.5pt);
      \fill (1.6,0.8) circle (0.5pt);

      \draw [decorate,decoration={brace,amplitude=5pt},xshift=-2pt,yshift=0pt]
      (0,0) -- (0,2) node[midway,left,xshift=-3pt]{\tiny$r \rleft( t_1
        \rright)$};

      \draw [decorate,decoration={brace,amplitude=5pt},xshift=2pt,yshift=0pt]
      (2,2) -- (2,1) node[midway,right,xshift=3pt]{\tiny$l \rleft( t_2
        \rright)$};
    \end{tikzpicture}\qquad
    \begin{tikzpicture}
      \draw (0,2) -- (2,2);
      
      \draw[-latex] (0.3,2.1) -- (0,2.1) node[above]{\tiny$\displaystyle \lim_{t\to t_1 +} f(t)$};
      \draw[-latex] (1.7,2.1) -- (2,2.1) node[above]{\tiny$\displaystyle \lim_{t\to t_2-}f(t)$};

      \draw[dashed] (0,2) -- (0,1);
      \draw[dashed] (2,2) -- (2,0);
      \fill (0,1) circle (1pt) node[below]{\tiny$f \rleft( t_1 \rright)$};
      \fill (2,0) circle (1pt) node[below]{\tiny$f \rleft( t_2 \rright)$};

      \fill (1.5,0.25) circle (1pt) node[below,xshift=-10pt]{\tiny$h^*_{\Tau_-,\xi_k}
        \rleft( \vect{y}_k \rright)$};
      \fill (1,0.5) circle (1pt);
      \fill (0.8,0.6) circle (0.5pt);
      \fill (0.6,0.7) circle (0.5pt);
      \fill (0.4,0.8) circle (0.5pt);

      \draw [decorate,decoration={brace,amplitude=5pt},xshift=-2pt,yshift=0pt]
      (0,1) -- (0,2) node[midway,left,xshift=-3pt]{\tiny$r \rleft( t_1
        \rright)$};

      \draw [decorate,decoration={brace,amplitude=5pt},xshift=2pt,yshift=0pt]
      (2,2) -- (2,0) node[midway,right,xshift=3pt]{\tiny$l \rleft( t_2
        \rright)$};
    \end{tikzpicture}
    \caption{The possible cases for two successive potential jump
      discontinuities.}
    \label{fig:succ-jd}
  \end{figure}
\end{proof}

\begin{ex}[Continuation of Example~\ref{ex:change-triang}]
  \label{ex:step-fct}
  The lattice polytope in Example~\ref{ex:change-triang} is an empty circuit
  with unique dependence relation
  \begin{align*}
    \vv_1 + 3 \rleft( \vv_2 + \vv_3 + \vv_4 \rright) + 
    2\rleft( \vv_5 + \vv_6 \rright) = 13 \vv_7 + \vv_8 \text{.}
  \end{align*}
  Let $\Am \coloneqq \set{ \vv_1, \ldots, \vv_8}$ be the associated vector set
  of homogenized lattice points in $\{1\} \times P$. We have $\Am_+ \coloneqq \{
  \vv_1, \ldots, \vv_6 \}$ and $\Am_- \coloneqq \{ \vv_7, \vv_8 \}$. In the
  proof of Lemma~\ref{lem:fliplemma-crk-1}, the coefficients of this dependence
  relation were denoted by $\rleft( \lambda_\vv \rright)_{\vv \in \Am_+}$ and
  $\rleft( \mu_\vv \rright)_{\vv \in \Am_-}$ respectively. There we also
  introduced $\vv'' \in \Am_+$ such that $x_\vv -
  \tfrac{x_{\vv''}}{\lambda_{\vv''}}\lambda_\vv \ge 0$ for all $\vv \in
  \Am_+$. In this example we have $\vv'' = \vv_1$. The periodic bounded step
  function (with period $1$) associated to $\vx$ is given as follows (see
  Figure~\ref{fig:bsf})
  \begin{align*}
    f \colon \R \to \Z; t \mapsto \lFrac{\tfrac{1}{5}-t} + 3\lFrac{\tfrac{4}{5}-3t}
    +2\lFrac{\tfrac{3}{5}-2t} + \lFrac{\tfrac{1}{5}+13t} + \lFrac{t} \text{.}
  \end{align*}
  
  Let us consider, \eg the possible jump discontinuity $t_{0}
  =\tfrac{3}{5}=0.6$. From Figure~\ref{fig:bsf} we can read off $f \rleft( t_{0}
  \rright)=2$ while $\lim_{t \to t_{0}-} f(t) = 3$ and $\lim_{t\to t_{0}+}f(t) =
  5$ which implies that $l \rleft( t_{0} \rright) = 1$ and $r \rleft( t_{0}
  \rright) = 3$.

  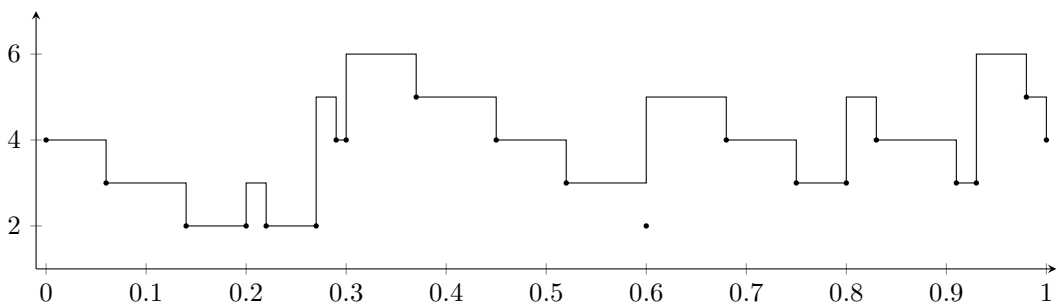
\begin{figure}[!ht]
    \centering
    \begin{tikzpicture}
      \begin{axis}[width=\textwidth,height=5cm,axis lines=left,enlarge x
        limits=0.01,enlarge y limits = 0.25]
        \addplot[mark=none] table[y=y] {\mydata};
        \fill (axis cs:0,4) circle (1pt);
        \fill (axis cs:0.06,3) circle (1pt);
        \fill (axis cs:0.14,2) circle (1pt);
        \fill (axis cs:0.2,2) circle (1pt);
        \fill (axis cs:0.22,2) circle (1pt);
        \fill (axis cs:0.27,2) circle (1pt);
        \fill (axis cs:0.29,4) circle (1pt);
        \fill (axis cs:0.30,4) circle (1pt);
        \fill (axis cs:0.37,5) circle (1pt);
        \fill (axis cs:0.45,4) circle (1pt);
        \fill (axis cs:0.52,3) circle (1pt);
        \fill (axis cs:0.60,2) circle (1pt);
        \fill (axis cs:0.68,4) circle (1pt);
        \fill (axis cs:0.75,3) circle (1pt);
        \fill (axis cs:0.80,3) circle (1pt);
        \fill (axis cs:0.83,4) circle (1pt);
        \fill (axis cs:0.91,3) circle (1pt);
        \fill (axis cs:0.93,3) circle (1pt);
        \fill (axis cs:0.98,5) circle (1pt);
        \fill (axis cs:1,4) circle (1pt);
      \end{axis}
    \end{tikzpicture}
    \caption{The periodic bounded step function of
      Example~\ref{ex:step-fct}. The dots indicate the value of $f$ at the
      potential jump discontinuities.}
    \label{fig:bsf}
  \end{figure}
  Finally, we have $f(0) = h^*_{\Tau_-,\Xi_C} \rleft( \vect{x} \rright) = 4$
  while $f \rleft( \tfrac{1}{5} \rright) = h^*_{\Tau_+,\Xi_C} \rleft( \vect{x}
  \rright) = 2$, and thus there is a gap at $3$. We can fill it by looking at
  the potential jump discontinuity $\tfrac{4}{65}$: $f \rleft( \tfrac{4}{65}
  \rright) = h^*_{\Tau_-,\xi} \rleft( \vect{x} - \vect{v}_7 \rright) = 3$ for
  $\xi \in \Xi_C \cap \sigma^\circ$ where $\sigma \in \Tau_-$ is the unique cone
  such that $\vect{v}_7$ does not generate a ray of it.
\end{ex}

\bibliographystyle{amsalpha}
\bibliography{nog}

\newcommand{\etalchar}[1]{$^{#1}$}
\def\cprime{$'$}
\providecommand{\bysame}{\leavevmode\hbox to3em{\hrulefill}\thinspace}
\providecommand{\MR}{\relax\ifhmode\unskip\space\fi MR }
\providecommand{\MRhref}[2]{%
  \href{http://www.ams.org/mathscinet-getitem?mr=#1}{#2}
}
\providecommand{\href}[2]{#2}
\begin{thebibliography}{AGH{\etalchar{+}}16}

\bibitem[AGH{\etalchar{+}}16]{Poly16}
B.~Assarf, E.~Gawrilow, K.~Herr, M.~Joswig, B.~Lorenz, A.~Paffenholz, and
  T.~Rehn, \emph{Polymake: tutorial for lattice polytopes},
  \url{https://polymake.org/doku.php/tutorial/lattice_polytopes_tutorial},
  2016.

\bibitem[Bat06]{Bat06}
V.~V. Batyrev, \emph{Lattice polytopes with a given {$h^*$}-polynomial},
  Algebraic and geometric combinatorics, Contemp. Math., vol. 423, Amer. Math.
  Soc., Providence, RI, 2006, pp.~1--10.

\bibitem[BB09]{BB:BoundStep}
J.~P. Bell and J.~W. Bober, \emph{Bounded step functions and factorial ratio
  sequences}, Int. J. Number Theory \textbf{5} (2009), no.~8, 1419--1431.

\bibitem[Bec16]{Beck14}
M.~Beck, \emph{{Stanley's major contributions to Ehrhart theory}}, {The
  mathematical legacy of Richard P. Stanley.} (P.~{Hersh}, T.~{Lam},
  P.~{Pylyavskyy}, and V.~{Reiner}, eds.), Providence, RI: American
  Mathematical Society (AMS), 2016, pp.~pp. 53--63.

\bibitem[BG09]{BG}
W.~Bruns and J.~Gubeladze, \emph{Polytopes, rings, and {$K$}-theory}, Springer
  Monographs in Mathematics, Springer, Dordrecht, 2009.

\bibitem[BH93]{Bruns}
W.~Bruns and J.~Herzog, \emph{Cohen--{M}acaulay rings}, Cambridge Studies in
  Advanced Mathematics, vol.~39, Cambridge University Press, Cambridge, 1993.

\bibitem[BM85]{Betke}
U.~Betke and P.~McMullen, \emph{Lattice points in lattice polytopes}, Monatsh.
  Math. \textbf{99} (1985), no.~4, 253--265.

\bibitem[BN07]{BN07}
V.~V. Batyrev and B.~Nill, \emph{Multiples of lattice polytopes without
  interior lattice points}, Mosc. Math. J. \textbf{7} (2007), no.~2, 195--207,
  349.

\bibitem[Bob09]{Bober:FracRat}
J.~W. Bober, \emph{Factorial ratios, hypergeometric series, and a family of
  step functions}, J. Lond. Math. Soc. \textbf{79} (2009), no.~2, 422--444.

\bibitem[Bor08]{Borisov:QuotSing}
A.~Borisov, \emph{Quotient singularities, integer ratios of factorials, and the
  {R}iemann hypothesis}, Int. Math. Res. Not. (2008), no.~15.

\bibitem[BR07]{BR:Computing}
M.~Beck and S.~Robins, \emph{Computing the continuous discretely.
  {I}nteger-point enumeration in polyhedra}, 2nd edition ed., Undergraduate
  Texts in Mathematics, Springer, New York, 2007.

\bibitem[Bra16]{Braun2016}
B.~Braun, \emph{Unimodality problems in {E}hrhart theory}, pp.~687--711,
  {Springer International Publishing}, 2016.

\bibitem[Bre15]{Breuer15}
F.~Breuer, \emph{An invitation to {E}hrhart theory: polyhedral geometry and its
  applications in enumerative combinatorics}, Computer algebra and polynomials,
  Lecture Notes in Comput. Sci., vol. 8942, Springer, 2015, pp.~1--29.

\bibitem[Bru13]{Bru13}
W.~Bruns, \emph{The quest for counterexamples in toric geometry}, Commutative
  algebra and algebraic geometry ({CAAG}-2010), Ramanujan Math. Soc. Lect.
  Notes Ser., vol.~17, Ramanujan Math. Soc., Mysore, 2013, pp.~45--61.

\bibitem[BS13]{BS}
M.~P. Brodmann and R.~Y. Sharp, \emph{Local cohomology}, second ed., Cambridge
  Studies in Advanced Mathematics, vol. 136, Cambridge University Press,
  Cambridge, 2013, An algebraic introduction with geometric applications.

\bibitem[BSV16]{Blek16}
G.~Blekherman, G.~G. Smith, and M.~Velasco, \emph{Sums of squares and varieties
  of minimal degree}, J. Amer. Math. Soc. \textbf{29} (2016), no.~3, 893--913.

\bibitem[Con02]{Con02}
H.~Conrads, \emph{Weighted projective spaces and reflexive simplices},
  Manuscripta Math. \textbf{107} (2002), no.~2, 215--227.

\bibitem[DdP09]{Dickenstein-Sandra}
A.~{Dickenstein}, S.~{di Rocco}, and R.~{Piene}, \emph{{Classifying smooth
  lattice polytopes via toric fibrations.}}, {Adv. Math.} \textbf{222} (2009),
  no.~1, 240--254.

\bibitem[DHNP13]{Adjunction-Paper}
S.~{Di Rocco}, C.~{Haase}, B.~{Nill}, and A.~{Paffenholz}, \emph{{Polyhedral
  adjunction theory.}}, {Algebra Number Theory} \textbf{7} (2013), no.~10,
  2417--2446.

\bibitem[DLRS10]{DRS:Triangulations}
J.~A. De~Loera, J.~Rambau, and F.~Santos, \emph{Triangulations}, Algorithms and
  Computation in Mathematics, vol.~25, Springer-Verlag, Berlin, 2010,
  Structures for algorithms and applications.

\bibitem[DN10]{Dickenstein-Nill}
A.~{Dickenstein} and B.~{Nill}, \emph{{A simple combinatorial criterion for
  projective toric manifolds with dual defect.}}, {Math. Res. Lett.}
  \textbf{17} (2010), no.~3, 435--448.

\bibitem[EG84]{EG}
D.~Eisenbud and S.~Goto, \emph{Linear free resolutions and minimal
  multiplicity}, J. Algebra \textbf{88} (1984), no.~1, 89--133.

\bibitem[Ehr62]{Ehr62}
E.~Ehrhart, \emph{Sur les poly{\`e}dres rationnels homoth{\'e}tiques {\`a}
  {$n$}\ dimensions}, C. R. Acad. Sci. Paris \textbf{254} (1962), 616--618.

\bibitem[Eis95]{eisenbud}
D.~Eisenbud, \emph{Commutative algebra. {W}ith a view toward algebraic
  geometry.}, Graduate Texts in Mathematics, vol. 150, Springer-Verlag, New
  York, 1995.

\bibitem[Est10]{Est10}
A.~Esterov, \emph{Newton polyhedra of discriminants of projections}, Discrete
  Comput. Geom. \textbf{44} (2010), no.~1, 96--148.

\bibitem[GJ00]{polymake}
E.~Gawrilow and M.~Joswig, \emph{polymake: a framework for analyzing convex
  polytopes}, Polytopes---combinatorics and computation ({O}berwolfach, 1997),
  DMV Sem., vol.~29, Birkh\"auser, Basel, 2000, pp.~43--73.

\bibitem[GKZ08]{GKZ}
I.~M. Gelfand, M.~M. Kapranov, and A.~V. Zelevinsky, \emph{Discriminants,
  resultants and multidimensional determinants}, Modern Birkh\"auser Classics,
  Birkh\"auser Boston, Inc., Boston, MA, 2008, Reprint of the 1994 edition.

\bibitem[Gub12]{Gub12}
J.~Gubeladze, \emph{Convex normality of rational polytopes with long edges},
  Adv. Math. \textbf{230} (2012), no.~1, 372--389.

\bibitem[HHN11]{Hibi11}
T.~Hibi, A.~Higashitani, and Y.~Nagazawa, \emph{Ehrhart polynomials of convex
  polytopes with small volumes}, European J. Combin. \textbf{32} (2011), no.~2,
  226--232.

\bibitem[Hib94]{hibilower}
T.~Hibi, \emph{A lower bound theorem for {E}hrhart polynomials of convex
  polytopes}, Advances in Mathematics \textbf{105} (1994), no.~2, 162--165.

\bibitem[HNP09]{HNP09}
C.~Haase, B.~Nill, and S.~Payne, \emph{Cayley decompositions of lattice
  polytopes and upper bounds for {$h^*$}-polynomials}, J. Reine Angew. Math.
  \textbf{637} (2009), 207--216.

\bibitem[HNP12]{HNP:LattPoly}
C.~Haase, B.~Nill, and A.~Paffenholz, \emph{Lecture notes on lattice
  polytopes},
  \url{https://polymake.org/polytopes/paffenholz/data/preprints/ln_lattice_polytopes.pdf},
  December 2012.

\bibitem[HT09]{HT:LowerBounds}
M.~Henk and M.~Tagami, \emph{Lower bounds on the coefficients of {E}hrhart
  polynomials}, European J. Combin. \textbf{30} (2009), no.~1, 70--83.

\bibitem[Ito15]{Ito15}
A.~Ito, \emph{Algebro-geometric characterization of {C}ayley polytopes}, {Adv.
  Math.} \textbf{270} (2015), 598--608.

\bibitem[Kas09]{Kas09}
A.~M. Kasprzyk, \emph{Bounds on fake weighted projective space}, Kodai Math. J.
  \textbf{32} (2009), no.~2, 197--208.

\bibitem[Kat15]{depthholes}
L.~Katth{\"a}n, \emph{Non-normal affine monoid algebras}, Manuscripta Math.
  \textbf{146} (2015), no.~1-2, 223--233.

\bibitem[KV08]{KV:Computing}
M.~K{\"o}ppe and S.~Verdoolaege, \emph{Computing parametric rational generating
  functions with a primal {B}arvinok algorithm}, Electron. J. Combin.
  \textbf{15} (2008), no.~1, Research Paper 16.

\bibitem[LZ91]{LZ91}
J.~C. Lagarias and G.~M. Ziegler, \emph{Bounds for lattice polytopes containing
  a fixed number of interior points in a sublattice}, Canad. J. Math.
  \textbf{43} (1991), no.~5, 1022--1035.

\bibitem[MP17]{MP:Counterexample}
J.~McCullough and I.~Peeva, \emph{Counterexamples to the {E}isenbud-{G}oto
  regularity conjecture}, J. Amer. Math. Soc. (2017), in press.

\bibitem[Nil08]{Nil08}
B.~Nill, \emph{Lattice polytopes having {$h^\ast$}-polynomials with given
  degree and linear coefficient}, European J. Combin. \textbf{29} (2008),
  no.~7, 1596--1602.

\bibitem[NP15]{Arnau}
B.~Nill and A.~Padrol, \emph{The degree of point configurations: {E}hrhart
  theory, {T}verberg points and almost neighborly polytopes}, European J.
  Combin. \textbf{50} (2015), 159--179.

\bibitem[Oga13]{Ogata}
S.~Ogata, \emph{Very ample but not normal lattice polytopes}, Beitr. Algebra
  Geom. \textbf{54} (2013), no.~1, 291--302.

\bibitem[OH06]{OH:Gorenstein}
H.~{Ohsugi} and T.~{Hibi}, \emph{{Special simplices and Gorenstein toric
  rings.}}, {J. Comb. Theory, Ser. A} \textbf{113} (2006), no.~4, 718--725.

\bibitem[SS90]{ss}
U.~Sch{\"a}fer and P.~Schenzel, \emph{Dualizing complexes of affine semigroup
  rings}, Trans. Amer. Math. Soc. \textbf{322} (1990), no.~2, 561--582.

\bibitem[Sta80]{Sta80}
R.~P. Stanley, \emph{Decompositions of rational convex polytopes}, Ann.
  Discrete Math. \textbf{6} (1980), 333--342, Combinatorial mathematics,
  optimal designs and their applications (Proc. Sympos. Combin. Math. and
  Optimal Design, Colorado State Univ., Fort Collins, Colo., 1978).

\bibitem[{Sta}89]{Stanley:Log-concave}
R.~P. {Stanley}, \emph{{Log-concave and unimodal sequences in algebra,
  combinatorics, and geometry.}}, {Graph theory and its applications: East and
  West. Proceedings of the first China-USA international conference, held in
  Jinan, China, June 9- 20, 1986}, New York: New York Academy of Sciences,
  1989, pp.~500--535.

\bibitem[SVL13]{Schepers13}
J.~Schepers and L.~Van~Langenhoven, \emph{Unimodality questions for integrally
  closed lattice polytopes}, Ann. Comb. \textbf{17} (2013), no.~3, 571--589.

\bibitem[Vas99]{Vasyunin}
V.~I. Vasyunin, \emph{On a system of step functions}, Zap. Nauchn. Sem.
  S.-Peterburg. Otdel. Mat. Inst. Steklov. (POMI) \textbf{262} (1999),
  no.~Issled. po Linein. Oper. i Teor. Funkts. 27, 49--70, 231--232.

\end{thebibliography}

\end{document}